\documentclass[letterpaper,10pt]{amsart}
\usepackage{amssymb}
\usepackage{xcolor}
\usepackage{hyperref}
\usepackage{dsfont}
\usepackage{enumerate}
\usepackage{mathrsfs}
\usepackage{mathtools}

\definecolor{carminepink}{rgb}{0.92, 0.3, 0.26}

\hypersetup{backref,colorlinks=true,allcolors=carminepink} 

\newtheorem{prop}{Proposition}[section]
\newtheorem{proposition}{Proposition}[section]

\newtheorem{lemma}[prop]{Lemma}
\newtheorem{definition}[prop]{Definition}
\newtheorem{corollary}[prop]{Corollary}
\newtheorem{theorem}[prop]{Theorem}

\newtheorem{remark}[prop]{Remark}

\newtheorem{example}[prop]{Example}

\DeclareSymbolFont{bbold}{U}{bbold}{m}{n}
\DeclareSymbolFontAlphabet{\mathbbold}{bbold}

\definecolor{electricultramarine}{rgb}{0.25, 0.0, 1.0}

\providecommand{\norm}[1]{\left\Vert#1\right\Vert}

\newcommand{\cont}[1]{\stackrel{#1}{\frown}}
\newcommand{\conts}[1]{\stackrel{#1}{\frown}}
\newcommand{\smcont}[1]{\star_{#1}^{#1 -1}}

\newcommand{\converge}{\underset{k \rightarrow +\infty}{\longrightarrow}}
\newcommand{\R}{\mathbb{R}}
\newcommand{\cvlaw}{\overset{\rm law}{\rightarrow}}

\newtheorem*{theorem*}{Theorem}
\newtheorem*{acknow*}{Acknowledgments}

\begin{document}

\author{Solesne Bourguin}
\address{Boston University, Department of Mathematics and Statistics, 111
  Cummington Mall, Boston, MA 02215, USA}
\email{bourguin@math.bu.edu}
\author{Ivan Nourdin}
\address{Universit\'e du Luxembourg, Maison du Nombre, 6 avenue de la Fonte, L-4364 Esch-sur-Alzette, Grand Duchy of Luxembourg}
\email{ivan.nourdin@uni.lu}
\title[Freeness characterizations on free chaos spaces]{Freeness characterizations on free chaos spaces}  
\begin{abstract}
This paper deals with characterizing the freeness and asymptotic
freeness of free multiple integrals with respect to a free Brownian
motion or a free Poisson process. We obtain three characterizations of
freeness, in terms of contraction operators, covariance conditions,
and free Malliavin gradients. We show how these characterizations can
be used in order to obtain limit theorems, transfer principles, and
asymptotic properties of converging sequences.
\end{abstract}
\maketitle

\section{Introduction}

A classical result in probability theory asserts that one can decompose any functional of a Brownian motion $W$ as an infinite sum of multiple integrals.
That is, to any square integrable random variable $F$ measurable with
respect to $W$, one can associate a unique sequence of symmetric and
square integrable kernels $\left\{ f_n \colon n \geq 0 \right\}$ such
that 
\begin{equation*}
F=\sum_{n=0}^\infty I_n^W(f_n).
\end{equation*}
The set of all multiple Wiener-It\^o integrals of the form $I_n^W(f)$, the so-called {\sl $n$-th Wiener chaos} of $W$, thus plays a fundamental role in modern stochastic analysis. Analysing its many rigid properties (notably those related to independence and normal approximation) has become a subject in its own right, and has grown into a mature and widely applicable mathematical theory.

Among the most striking results about Wiener chaos are the following two theorems, which will play a central role in the present paper.
The first one characterizes independence of multiple Wiener-It\^o integrals.
\begin{theorem}[{\"U}st{\"u}nel and Zakai \cite{ustunel_independence_1989}, 1989]\label{ustu-zakai}
Let $n,m$ be natural numbers and let $f\in L^2(\R_+^n)$ and $g\in L^2(\R_+^m)$ be symmetric functions. Then $I_n^W(f)$ and $I_m^W(g)$ are independent if and only if, for almost all $x_1,\ldots,x_{n-1},y_1,\ldots,y_{m-1}\in\R_+$, $$\int_0^\infty f(x_1,\ldots,x_{n-1},u)g(y_1,\ldots,y_{m-1},u)du=0.$$
\end{theorem}
The second result is nowadays one of the most central tools of analysis on Wiener chaos, as it represents a drastic simplification with respect to the method of moments for the normal approximation of sequences of multiple Wiener-It\^o integrals. 

\begin{theorem}[Nualart and Peccati \cite{nualart_central_2005}, 2005]\label{nunu-gio}
A unit-variance sequence in a Wiener chaos of fixed order converges in law to the standard Gaussian distribution if and only if the corresponding sequence of fourth moments converges to three.
\end{theorem}

Since its introduction by Voiculescu in the eighties in order to solve some longstanding conjectures about von Neumann algebras of free groups, free probability theory has become a vivid and powerful branch of mathematics, with many applications (including signal processing, chanel capacity estimation and nuclear physics) and deep connections with other mathematical fields (like operator algebra, theory of random matrices or combinatorics).
Free probability has many parallels with the usual probability theory (hence its name), and the study of these links often brings a new point of view which may then enrich the theory of both worlds (classical and free).

Starting from the free independence property, a genuine stochastic
calculus with respect to the free Brownian motion (the free analogue
of the classical Brownian motion) has emerged within the last twenty
years, following the route paved by the seminal paper of Biane and
Speicher \cite{biane_stochastic_1998}. In particular, a common
property of the classical and free settings is the possibility of expanding the space as a sum of free chaos, giving rise to the so-called {\sl Wigner chaos}. 
By their very construction, these free chaos play in the free world
a similar role as Wiener chaos in the classical setting.
It is thus natural to investigate the similarities and differences between these two mathematical objects.
For instance, do we have an analogue of Theorem \ref{nunu-gio} in the free world? The answer is yes, and is given by the following theorem
taken from \cite{kemp_wigner_2012}.

\begin{theorem}[Kemp et. al \cite{kemp_wigner_2012}, 2012]\label{knps-intro}
A unit-variance sequence in a Wigner chaos of fixed order converges in law to the semicircular distribution if and only if the corresponding sequence of fourth moments converges to 2.
\end{theorem}

Shortly after the publication of \cite{kemp_wigner_2012}, many other
results in the spirit of Theorem \ref{knps-intro} have been added to
the literature, including the following ones (the list is not exhaustive).

In
\cite{nourdin_multi-dimensional_2013}, it is shown that component-wise
convergence to the semicircular distribution is equivalent to joint
convergence, thus extending to the free probability  setting a seminal
result by Peccati and Tudor (see also
\cite{peccati_gaussian_2005}).

In \cite{nourdin_poisson_2013}, a
non-central counterpart of Theorem \ref{knps-intro} is provided. More
precisely, it is shown that any adequately rescaled sequence $\left\{
  F_n \colon n\geq 0 \right\}$ of self-adjoint operators living inside
a fixed Wigner chaos of even order converges in distribution to a
centered free Poisson random variable with rate $\lambda>0$ if and
only if $\varphi(F_n^4)-2\varphi(F_n^3)\to 2\lambda^2-\lambda$ (where
$\varphi$ is the relevant tracial state).

In
\cite{nourdin_convergence_2012}, convergence in law of any sequence
belonging to the second Wigner chaos is characterized by means of the
convergence of only a finite number of cumulants.

In
\cite{deya_convergence_2012}, making use of heavy combinatorics it is
shown that any adequately rescaled sequence $\left\{ F_n \colon n\geq
  0 \right\}$ of self-adjoint operators living inside a fixed Wigner
chaos converges in distribution to the tetilla law $\mathcal{T}$ if
and only if $\varphi(F_n^4)\to \varphi(\mathcal{T}^4)$ and
$\varphi(F_n^6)\to \varphi(\mathcal{T}^6)$ (where $\varphi$ is the
relevant tracial state). Note that this finding is not an extension of
a result known in the classical probability theory, as the existence
of such a result in the classical setting is still an open problem.

In
\cite{bourguin_semicircular_2014}, a class of sufficient conditions,
ensuring that a sequence of multiple integrals with respect to a free
Poisson measure converges to a semicircular limit, is established,
thus providing an analog of Theorem \ref{knps-intro} in the context of free
Poisson chaos.

In \cite{bourguin_poisson_2015}, a fourth moment type
condition is given, for an element of a free Poisson chaos of
arbitrary order to converge to a free centered Poisson
distribution.

In \cite{arizmendi_convergence_2014}, an estimate for
the Kolmogorov distance between a freely infinitely divisible
distribution and the semicircle distribution is given, in terms of the
difference between the fourth moment and two.

In
\cite{bourguin_vector-valued_2016}, a multidimensional counterpart of
the aforementioned central limit theorem on the free Poisson chaos is
given.

In \cite{bourguin_free_2017}, a quantitative version of Theorem
 \ref{knps-intro} is derived, using free stochastic analysis as well as a new biproduct formula for bi-integrals.

In the present paper, our main goal is to provide characterizations of
free independence on the Wigner and free Poisson chaos, as well as
investigate the similarities and dissimilarities between classical and
free chaos, as far as (possibly asymptotic) independence properties are concerned. 

Our first set of investigations yields a characterization of freeness on the Wigner and free Poisson chaos, in terms of contractions, covariances, or free Malliavin gradient, thus providing a suitable extension of Theorem \ref{ustu-zakai} (and related results) to the free setting. Most of our results turn out to be similar to the classical setting, with the notable exception of the characterization of freeness in terms of the free Malliavin gradient, this last fact illustrating a fundamental difference between the classical and the free cases.

Our second set of investigations is concerned again with the independence property, but this time in an asymptotic context. Here, the problem is to find what conditions are to be imposed on {\sl limits} of multiple integrals to be free.

The remainder of this paper is organized as follows: Section 2 contains a short introduction to free probability theory, with a special emphasis to the material needed for the rest of the paper.
Section 3 is devoted to the characterization of freeness on the Wigner
and free Poisson chaos, in terms of contractions, covariances, or free
Malliavin gradient. This section also provides several lemmas which will be used to prove our main results in the following sections.
In Section 4, we study different characterizations of asymptotic freeness, in several contexts. 
We devote Section 5 to the study of transfer principles between
classical and free chaos.
Finally, Section 6 contains auxiliary results that are used throughout the paper.

\section{Preliminaries}

\subsection{Elements of free probability}
In the following, a short introduction to free probability theory is
provided. For a thorough and complete treatment, see \cite{nica_lectures_2006},
\cite{voiculescu_free_1992} and \cite{hiai_semicircle_2000}. Let $\left(
  \mathscr{A}, \varphi \right) $ be a tracial $W^*$-probability space, that is
$\mathscr{A}$ is a von Neumann algebra with involution $*$ and $\varphi \colon
\mathscr{A} \rightarrow \mathbb{C}$ is a unital linear functional assumed to be
weakly continuous, positive (meaning that $\varphi\left( X\right) \geq 0$
whenever  $X$ is a non-negative element of $\mathscr{A}$), faithful (meaning
that $\varphi\left(XX^* \right) = 0 \Rightarrow X = 0$ for every $X \in
\mathscr{A}$) and tracial (meaning that $\varphi\left(XY \right) =
\varphi\left(YX \right) $ for all $X,Y \in \mathscr{A}$). The self-adjoint
elements of $\mathscr{A}$ will be referred to as random variables.
The non-commutative space 
$L^2(\mathscr{A},\varphi)$ denotes the completion of $\mathscr{A}$ with respect
to the norm $\norm{X}_2 = \sqrt{\varphi\left( XX^* \right) }$.

Recall the definition of freeness (see \cite[Definition
5.3]{nica_lectures_2006} and \cite[Remarks 5.4]{nica_lectures_2006} or \cite[Definition 2.5.18]{tao_topics_2012}) for a collection of non-commutative random variables living on an appropriate non-commutative probability space $\left(\mathscr{A}, \varphi \right) $.
\begin{definition}
  \label{deffreeness}
A collection of random variables $X_1, \ldots , X_n$ on $\left( \mathscr{A}, \varphi \right) $ is said to be free if $$\varphi\left( \left[P_1\left(X_{i_1}\right) - \varphi\left( P_1\left(X_{i_1}\right)\right)   \right] \cdots \left[P_m\left(X_{i_m}\right) - \varphi\left( P_m\left(X_{i_m}\right)\right)   \right] \right) = 0 $$ whenever $P_1, \ldots , P_m$ are polynomials and $i_1, \ldots , i_m \in \left\lbrace 1, \ldots, n\right\rbrace $ are indices with no two adjacent $i_j$ equal.
\end{definition}

 Let $X\in \mathscr{A}$. The $p$-th moment of $X$ is given by the quantity
 $\varphi(X^{p})$, $p \in \mathbb{N}_0$. Now assume that $X$ is a
 self-adjoint bounded element of $\mathscr{A}$ (in other words, $X$ is
 a bounded random variable), and write $\rho(X)= \norm{X} \in [0, \infty)$ to indicate the {\it spectral radius} of $X$. 
\begin{definition}
The {\it law} (or {\it spectral measure}) of $X$ is defined as the
unique Borel probability measure $\mu_{X}$ on the real line such that $\int_{\mathbb{R}}P(t)\ d\mu_{X}(t) = \varphi(P(X))$
for every polynomial $P \in \mathbb{R}\left[ X\right]$. A consequence
of this definition is that $\mu_X$ has support in $[-\rho(X), \rho(X)]$.

\end{definition}
The existence and uniqueness of $\mu_X$ in such a general framework are proved e.g. in \cite[Theorem 2.5.8]{tao_topics_2012} (see also \cite[Proposition 3.13]{nica_lectures_2006}). Note that, since $\mu_X$ has compact support, the measure $\mu_X$ is completely determined by the sequence $\left\lbrace \varphi(X^p) \colon p\geq 1\right\rbrace $. 

Let $\left\lbrace X_{k} \colon k \geq 1\right\rbrace $ be a sequence of non-commutative random variables, each possibly belonging to a different non-commutative probability space $(\mathscr{A}_k, \varphi_k)$. 
\begin{definition}
The sequence $\left\lbrace X_{k} \colon k \geq 1\right\rbrace $ is said to converge in distribution to a limiting non-commutative random variable $X_{\infty}$ (defined on $(\mathscr{A}_\infty, \varphi_\infty)$), if $\varphi_k(P(X_{k})) \underset{k \rightarrow +\infty}{\longrightarrow} \varphi_\infty(P(X_{\infty}))$ for every polynomial $P\in \R[X]$.
\end{definition}
 If $X_{k}, X_\infty$ are bounded (and therefore the spectral measures
 $\mu_{X_k}, \mu_{X_\infty}$ are well-defined), this last relation is
 equivalent to saying that $$\int_\R P(t)\, \mu_{X_k}(dt) \underset{k \rightarrow +\infty}{\longrightarrow} \int_\R
 P(t)\, \mu_{X_\infty}(dt).$$ An application of the method of moments
 yields immediately that, in this case, one has also that $\mu_{X_k}$
 weakly converges to $\mu_{X_\infty}$, that is $\mu_{X_k}(f) \underset{k \rightarrow +\infty}{\longrightarrow}
 \mu_{X_\infty}(f)$, for every $f\colon \R\to \R$ bounded and continuous
 (note that no additional uniform boundedness assumption is needed).
 
In this paper, we will also deal with {\it joint} convergences in law, for sequences
$\left\lbrace X_{k}=(X_k^1,\ldots,X_k^d) \colon k \geq 1\right\rbrace $ of non--commutative random vectors, each possibly belonging to a different non-commutative probability space $(\mathscr{A}_k, \varphi_k)$. 
\begin{definition}
The vector-valued sequence $\left\lbrace X_{k}=(X_k^1,\ldots,X_k^d) \colon k \geq 1\right\rbrace $ is said to converge jointly in distribution to a limiting non-commutative random vector $X_{\infty}=(X_\infty^1,\ldots,X_\infty^d)$ (defined on $(\mathscr{A}_\infty, \varphi_\infty)$), if any moment in the variables $X_k^1,\ldots,X_k^d$ converges, as $k\to\infty$, to the corresponding moments in $X_\infty^1,\ldots,X_\infty^d$; otherwise stated, $(X_k^1,\ldots,X_k^d) \overset{\rm law}{\to} (X_\infty^1,\ldots,X_\infty^d)$ if for any $r\in\mathbb{N}$ and positive
integers $i_1,\ldots,i_r$, one has, as $k\to\infty$:
$$
\varphi_k\big[X_k^{i_1}\ldots X_k^{i_r}\big]\to \varphi_\infty\big[X_\infty^{i_1}\ldots X_\infty^{i_r}\big].
$$
\end{definition}
 
 Let us now define the two main processes we will deal with in this paper, namely the free Brownian motion and the free Poisson process.
 \begin{definition}
 \begin{enumerate}[1.]
\item The centered semicircular distribution with variance $t>0$, denoted by
$\mathcal{S}(0,t)$, is the probability distribution given
by $$\mathcal{S}(0,t)(dx) = (2\pi t)^{-1}\sqrt{4t-x^2}
\mathds{1}_{\left[ -2\sqrt{t} , 2\sqrt{t} \right]}(x)dx.$$ 
\item A free Brownian motion $S$ consists of: (i) a filtration $\left\lbrace \mathscr{A}_t \colon t \geq 0 \right\rbrace $ of von Neumann sub-algebras of $\mathscr{A}$ (in particular, $\mathscr{A}_s \subset \mathscr{A}_t$ for $0 \leq s < t$), (ii) a collection $S = \left\lbrace S_t \colon t\geq 0\right\rbrace $ of self-adjoint operators in $\mathscr{A}$ such that: (a) $S_0 = 0$ and $S_t \in \mathscr{A}_t$ for all $t \geq 0$, (b) for all $t \geq 0$, $S_t$ has a semicircular distribution with mean zero and variance $t$, and (c) for all $0 \leq u < t$, the increment $S_t - S_u$ is free with respect to $\mathscr{A}_u$, and has a semicircular distribution with mean zero and variance $t-u$.
\end{enumerate}
\end{definition}
\begin{definition}
 \begin{enumerate}[1.]
\item The free Poisson distribution with rate $\lambda >0 $, denoted by $P(\lambda)$, is the probability distribution defined 
as follows: (i) if $\lambda\in (0,1]$, then $P(\lambda) = (1-\lambda)\delta_0 + \lambda\widetilde{\nu}$, and (ii) if $\lambda > 1$, then 
$P(\lambda) = \widetilde{\nu}$, where $\delta_0$ stands for the Dirac
mass at $0$. Here,
\begin{equation*}
\widetilde{\nu}(dx) = 
(2\pi x)^{-1}\sqrt{4\lambda-(x-1-\lambda)^2}
\mathds{1}_{\left[(1-\sqrt{\lambda})^2 , (1+\sqrt{\lambda})^2  \right]}(x)dx.
\end{equation*}
\item A free Poisson process $N$ consists of: (i) a filtration $\left\lbrace \mathscr{A}_t \colon t \geq 0 \right\rbrace $ of von Neumann sub-algebras of $\mathscr{A}$ (in particular, $\mathscr{A}_s \subset \mathscr{A}_t$ for $0 \leq s < t$), (ii) a collection $N = \left\lbrace N_t \colon t\geq 0\right\rbrace $ of self-adjoint operators in $\mathscr{A}_{+}$ ($\mathscr{A}_{+}$ denotes the cone of positive operators in $\mathscr{A}$) such that: (a) $N_0 = 0$ and $N_t \in \mathscr{A}_t$ for all $t \geq 0$, (b) for all $t \geq 0$, $N_t$ has a free Poisson distribution with rate $t$, and (c) for all $0 \leq u < t$, the increment $N_t - N_u$ is free with respect to $\mathscr{A}_u$, and has a free Poisson distribution with rate $t-u$. $\hat N$ will denote the collection of random variables $\hat{N}= \left\lbrace  \hat{N}_t = N_t - t\mathbf{1} \colon t\geq 0 \right\rbrace $, where $\mathbf{1}$ stands for the unit of $\mathscr{A}$. $\hat N$ will be referred to as a compensated free Poisson process.
\end{enumerate}
\end{definition}
\begin{remark}
{\rm
 In the sequel, $\frak{M}$ will stand for either the free Brownian motion $S$ or the compensated free Poisson process $\hat{N}$. 
 }
\end{remark}

We continue with some definitions that will play a crucial role in the rest of the paper.
For every integer $n\geq 1$, the space $L^2\left( \R_{+}^n;\mathbb{C}\right) = L^2\left( \R_{+}^n\right)$ denotes the collection of all complex-valued functions on $\R_{+}^n$ that are square-integrable with respect to the Lebesgue measure on $\R_{+}^n$. 
\begin{definition}
\label{mirrorsymdef}
Let $n$ be a natural number and let $f$ be a function in $L^2\left( \R_{+}^n\right)$. 
\begin{enumerate}[1.]
\item The adjoint of $f$ is the function $f^{\ast}\left(t_1, \ldots , t_n \right) = \overline{f\left(t_n, \ldots , t_1\right)}$.
\item The function $f$ is called mirror-symmetric if $f = f^{\ast}$, i.e., if $$f\left(t_1, \ldots , t_n \right) = \overline{f\left(t_n, \ldots , t_1\right)}$$ for almost all $\left( t_1,\ldots , t_n\right) \in \R_{+}^{n}$ with respect to the product Lebesgue measure.
\item The function $f$ is called (fully) symmetric if it is real-valued and, for any permutation $\sigma$ in the symmetric group $\mathfrak{S}_{n}$, it holds that $f\left( t_1, \ldots , t_n\right) = f\left( t_{\sigma(1)}, \ldots , t_{\sigma(n)}\right) $ for almost all $\left( t_1,\ldots , t_n\right) \in \R_{+}^{n}$ with respect to the product Lebesgue measure.
\end{enumerate}
\end{definition}
\begin{definition}
\label{defcontractions}
Let $n,m$ be natural numbers and let $f \in L^2\left( \R_{+}^n\right)$ and $g \in L^2\left( \R_{+}^m\right)$. Let $p \leq n \wedge m$ be a natural number. The $p$-th nested contraction $f \cont{p} g$ of $f$ and $g$ is the $L^2\left( \R_{+}^{n+m-2p}\right)$ function defined by nested integration of the middle $p$ variables in $f \otimes g$:
\begin{eqnarray*}
( f  \cont{p} g) (t_1,\ldots, t_{n+m - 2p}) &=& \int_{\R_{+}^{p}}f(t_1,\ldots , t_{n-p},s_1, \ldots , s_p) \\
&& \qquad \times g(s_p, \ldots , s_1 , t_{n-p+1},\ldots, t_{n+m-2p})ds_1 \cdots ds_p.
\end{eqnarray*}
In the case where $p=0$, the function $f  \cont{0} g$ is just given by
$f \otimes g$.
\end{definition}
Similarly, we define the star contraction $f \star_k^j g$ of $f$ and $g$.  \begin{definition}
\label{defcontractions2}
Let $n,m$ be natural numbers and let $f \in L^2\left( \R_{+}^n\right)$ and $g \in L^2\left( \R_{+}^m\right)$. Let $k \in \{1,\ldots, n \wedge m\}$ and $j\in\{0,\ldots,k\}$ be two natural numbers. We set
\begin{align*}
&( f  \star_k^j g) (t_1,\ldots, t_{n+m - 2k+j})= \int_{\R_{+}^{k-j}}f(t_1,\ldots , t_{n-k+j},s_{k-j}, \ldots , s_{1}) \\
&\hskip3.5cm  \times g(s_{1}, \ldots , s_{k-j} , t_{n-k+1},\ldots, t_{n+m-2k+j})ds_1 \cdots ds_{k-j}.
\end{align*}
\end{definition}

 For $f \in L^2\left(\R_{+}^n\right)$, we denote by $I_{n}^S(f)$ the multiple
 Wigner integral of $f$ with respect to the free Brownian motion as introduced
 in \cite{biane_stochastic_1998}. The space $L^2(\mathcal{S}, \varphi) =
 \{I_n^S(f) \colon f\in L^2(\R_{+}^n), n\geq 0\}$ is a unital $\ast$-algebra, with
 product rule given, for any $n,m\geq 1$, $f \in L^2\left(\R_{+}^n\right)$, $g
 \in L^2\left(\R_{+}^m\right)$, by 
\begin{equation}
\label{productformula}
I_n^S(f)I_m^S(g) = \sum_{p=0}^{n \wedge m} I_{n+m-2p}^S\left( f \cont{p} g\right) 
\end{equation}
and involution $I_n^S(f)^{\ast} = I_n^S(f^{\ast})$. For a proof of this
formula, see \cite{biane_stochastic_1998}.

Similarly, we can define
free Poisson multiple integrals with respect to $\hat{N}$ (these
integrals were studied in depth in \cite{bourguin_semicircular_2014},
and we refer to this reference for details). The space
$L^2(\mathcal{N}, \varphi) = \{I^{\hat{N}}_n(f) \colon f\in L^2(\R_{+}^n), n\geq 0\}$ is a unital $\ast$-algebra, with product rule given, for any $n,m\geq 1$, $f \in L^2\left(\R_{+}^n\right)$, $g \in L^2\left(\R_{+}^m\right)$, by
\begin{equation}
  \label{productformulafreepoisson}
I^{\hat{N}}_n(f)I^{\hat{N}}_m(g) = \sum_{p=0}^{n \wedge m} I^{\hat{N}}_{n+m-2p}\left( f \cont{p} g\right) +  \sum_{p=1}^{n \wedge m} I^{\hat{N}}_{m+n-2p+1}\left( f \smcont{p} g\right)
\end{equation}
and involution $I_n^{\hat{N}}(f)^{\ast} =
I_n^{\hat{N}}(f^{\ast})$. For a proof of this formula, see \cite{bourguin_semicircular_2014}.

Furthermore, as is well-known, both Wigner and free Poisson multiple integrals of different orders are orthogonal in $L^2(\mathscr{A},\varphi)$, whereas for two integrals of the same order, the Wigner isometry holds:
\begin{equation}
\label{wigneriso}
\varphi\left( I_n^{\frak{M}}(f)I_n^{\frak{M}}(g)^{*}\right) = \left\langle f,g \right\rangle_{L^2\left(\mathbb{R}_{+}^n \right) }.
\end{equation}

\begin{remark}
{\rm
\begin{enumerate}[1.]
\item
Observe that it follows from the definition of the involution on the
algebras $L^2(\mathcal{S}, \varphi)$ and $L^2(\mathcal{N}, \varphi)$ that operators of the type $I_n^{\frak{M}}(f)$ are self-adjoint if and only if $f$ is mirror-symmetric.
\item
In what follows, we will use the notation $I_n^S$, $I_n^{\hat{N}}$,
$I_n^W$ and $I_n^{\hat{\eta}}$ to denote multiple Wigner integrals,
multiple free Poisson integrals,
multiple Wiener integrals, and multiple classical Poisson integrals, respectively.
\end{enumerate}
}
\end{remark}

\subsection{Bi-integrals and free gradient operator}

In this particular subsection, we only focus on the Wigner case, as
the tools we are about to introduce do not exist in the context of
free Poisson processes.

Let $\left( \mathscr{A},\varphi\right) $ be a $W^{*}$-probability space. An $\mathscr{A} \otimes \mathscr{A}$-valued stochastic process $t\mapsto U_t$ is called a biprocess. For $p \geq 1$, $U$ is an element of $\mathscr{B}_{p}$, the space of $L^p$-biprocesses, if its norm
\begin{equation*}
\norm{U}_{\mathscr{B}_p}^{2} = \int_{0}^{\infty} \norm{U_t}_{L^p\left(\mathscr{A} \otimes \mathscr{A}, \varphi \otimes \varphi\right) }^{2}dt
\end{equation*}
is finite.

Let $n,m$ be two positive integers and $f=g\otimes h \in L^2\left( \mathbb{R}_{+}^{n}\right) \otimes L^2\left( \mathbb{R}_{+}^{m}\right)$. Then, the Wigner bi-integral $[I_n^S \otimes I_m^S](f)$ is defined as 
\begin{equation*}
[I_n^S \otimes I_m^S](f) = I_n^S(g) \otimes I_m^S(h).
\end{equation*}
From the Wigner isometry  for multiple integrals, we obtain the so called Wigner bisometry: for $f \in L^2\left( \mathbb{R}_{+}^{n}\right) \otimes L^2\left( \mathbb{R}_{+}^{m}\right)$ and $g \in L^2\left( \mathbb{R}_{+}^{n'}\right) \otimes L^2\left( \mathbb{R}_{+}^{m'}\right)$ having the form of a tensor product, it holds that
\begin{equation}
\label{wignerbisometry}
\varphi \otimes \varphi\left([I_n^S \otimes I_m^S](f) [I_{n'}^S \otimes I_{m'}^S](g)^{*} \right)= \begin{cases} \left\langle f,g \right\rangle_{L^2\left( \mathbb{R}_{+}^{n}\right) \otimes L^2\left( \mathbb{R}_{+}^{m}\right)} &  \text{if $n=n'$ and $m=m'$},\\ 0 &  \text{otherwise.}\end{cases} 
\end{equation}
Formula (\ref{wignerbisometry}) is then extended linearly to generic elements $f \in L^2\left(
  \mathbb{R}_{+}^{n}\right) \otimes L^2\left( \mathbb{R}_{+}^{m}\right) \cong
L^2\left( \mathbb{R}_{+}^{n+m}\right)$, where the symbol $\cong$ denotes an
isomorphic identification. 

A crucial tool in the analysis of Wigner integrals is the product
formula \eqref{productformula}, and a biproduct formula for bi-integrals was recently
obtained in \cite{bourguin_free_2017}, which will be a crucial
tool in the sequel. It makes use of a new type of
contraction, referred to in \cite{bourguin_free_2017} as
bicontractions, defined as follows.
Let $n_1, m_1, nf_2, m_2$ be positive integers. Let $f \in L^2\left(
  \mathbb{R}_{+}^{n_1}\right) \otimes L^2\left( \mathbb{R}_{+}^{m_1}\right)
\cong L^2\left( \mathbb{R}_{+}^{n_1+m_1}\right)$ and $g \in L^2\left(
  \mathbb{R}_{+}^{n_2}\right) \otimes L^2\left( \mathbb{R}_{+}^{m_2}\right)
\cong L^2\left( \mathbb{R}_{+}^{n_2+m_2}\right)$ and let $p \leq n_1 \wedge
n_2$, $r \leq m_1 \wedge m_2$ be natural numbers. The $(p,r)$-bicontraction $f
\conts{p,r} g$ is the $L^2\left( \mathbb{R}_{+}^{n_1+n_2 -2p}\right) \otimes
L^2\left( \mathbb{R}_{+}^{m_1+m_2 -2r}\right) \cong L^2\left(
  \mathbb{R}_{+}^{n_1+n_2+m_1+m_2 -2p-2r}\right)$ function defined by 
\begin{align*}
&f \conts{p,r} g(  t_1, \ldots , t_{n_1+n_2+m_1+m_2 -2p-2r} ) = \int_{\mathbb{R}_{+}^{p+r}}f(t_1, \ldots ,
 t_{n_1-p},s_p,\ldots,s_1, y_1,\ldots, y_r, \\
 & \qquad\qquad\qquad\qquad\qquad\qquad\qquad\qquad t_{n_1 + n_2 +m_2 -2p -r +1}, \ldots , t_{n_1 + n_2 + m_1 +m_2 -2p -2r}  ) \\  
 &\qquad \times g\left( s_1 , \ldots , s_p , t_{n_1-p+1},\ldots, t_{n_1+n_2+m_2 -2p-r}, y_r, \ldots, y_1\right)  ds_1 \cdots ds_p dy_1 \cdots dy_r.
\end{align*}
\begin{remark}
  \label{remarkbicont}
  {\rm
Observe that these bicontractions have the following properties (for a
proof, see \cite{bourguin_free_2017}). For $n_1, m_1, n_2, m_2 \in \mathbb{N}$, let $f \in L^2\left(
  \mathbb{R}_{+}^{n_1}\right) \otimes L^2\left( \mathbb{R}_{+}^{m_1}\right)
\cong L^2\left( \mathbb{R}_{+}^{n_1+m_1}\right)$ and $g \in L^2\left(
  \mathbb{R}_{+}^{n_2}\right) \otimes L^2\left( \mathbb{R}_{+}^{m_2}\right)
\cong L^2\left( \mathbb{R}_{+}^{n_2+m_2}\right)$ be fully symmetric
functions. Furthermore, let $p \leq n_1 \wedge n_2$ and $r \leq m_1 \wedge m_2$ be natural
numbers such that $p+r = p'+r'$. Then, the following holds.
\begin{enumerate}[1.]
\item $f \conts{p,r} g \cong f \conts{p+r} g$.
\item $f \conts{p,r} g =f \conts{p',r'} g$.
\item $\norm{f \conts{p,r} g}_{L^2\left( \mathbb{R}_{+}^{n_1+n_2 -2p}\right) \otimes L^2\left( \mathbb{R}_{+}^{m_1+m_2 -2r}\right)}^{2} = \norm{f \cont{p+r} g}_{L^2\left( \mathbb{R}_{+}^{n_1+n_2+m_1+m_2 -2p-2r}\right)}^{2}$.
\item $f \conts{n_1,m_1} f = \norm{f}_{L^2\left( \mathbb{R}_{+}^{n_1}\right) \otimes L^2\left( \mathbb{R}_{+}^{m_1}\right)}^{2} 1\otimes 1$, which is a constant in $L^2\left( \mathbb{R}_{+}^{n_1}\right) \otimes L^2\left( \mathbb{R}_{+}^{m_1}\right)$.
\end{enumerate}
}
\end{remark}

We introduce $\sharp$ to be the associative
 action of $\mathscr{A} \otimes \mathscr{A}^{\operatorname{op}}$ (where
 $\mathscr{A}^{\operatorname{op}}$ denotes the opposite algebra) on $\mathscr{A}
 \otimes \mathscr{A}$, as 
\begin{equation}
\label{sharpdef}
(A\otimes B) \sharp (C \otimes D) = (AC) \otimes (DB).
\end{equation}
Furthermore, we also write $\sharp$ to denote the action of $\mathscr{A} \otimes L^2\left( \mathbb{R}_{+}\right) \otimes  \mathscr{A}^{\operatorname{op}}$ on $\mathscr{A} \otimes L^2\left( \mathbb{R}_{+}\right) \otimes  \mathscr{A}$, as
\begin{equation*}
\label{sharpdefwithhilbert}
(A\otimes f \otimes  B) \sharp (C \otimes g \otimes  D) = (AC) \otimes fg \otimes  (DB).
\end{equation*}
Using the bicontractions definition, the biproduct formula for Wigner
bi-integrals proved in \cite{bourguin_free_2017} can be stated as follows.
\begin{proposition}[Bourguin and Campese \cite{bourguin_free_2017}, 2017]
\label{productformulabi}
For $n_1, m_1, n_2, m_2 \in \mathbb{N}$, let $f \in L^2\left(
  \mathbb{R}_{+}^{n_1}\right) \otimes L^2\left( \mathbb{R}_{+}^{m_1}\right)
\cong L^2\left( \mathbb{R}_{+}^{n_1+m_1}\right)$ and $g \in L^2\left(
  \mathbb{R}_{+}^{n_2}\right) \otimes L^2\left( \mathbb{R}_{+}^{m_2}\right)
\cong L^2\left( \mathbb{R}_{+}^{n_2+m_2}\right)$. Then it holds that 
\begin{equation}
\label{biintegralmultformula}
[I_{n_1}^S \otimes I_{m_1}^S]\left(f\right) \sharp [I_{n_2}^S \otimes I_{m_2}^S]\left(g\right)= \sum_{p=0}^{n_1 \wedge n_2}\sum_{r=0}^{m_1 \wedge m_2}[I_{n_1+n_2 -2p}^S \otimes I_{m_1+m_2 -2r}^S]\left(f \conts{p,r}g\right).
\end{equation}
\end{proposition}

Finally, the free gradient operator $\nabla \colon L^2\left(\mathcal{S},\varphi \right) \rightarrow \mathscr{B}_{2} $ is a densely-defined and closable operator whose action on Wigner integrals is given by
\begin{equation*}
\nabla_t I_n^S(f) = \sum_{k=1}^{n}[I_{k-1}^S \otimes I_{n-k}^S]\left(f_t^{(k)} \right),
\end{equation*}
where $f_t^{(k)}(x_1,\ldots,x_{n-1}) = f(x_1,\ldots,
x_{k-1},t,x_{k},\ldots, x_{n-1})$ is viewed as an element of
$L^2\left(\mathbb{R}_{+}^{k-1} \right) \otimes
L^2\left(\mathbb{R}_{+}^{n-k} \right)$.
We also
define the pairing $\left\langle \cdot, \cdot
\right\rangle$ between $\mathscr{B}_2 \times \mathscr{B}_2$ and $
L^2(\mathscr{A}\otimes \mathscr{A}, \varphi\otimes \varphi)$ to be 
\begin{align}
  \label{defofinnerprodint}
   \left\langle \cdot, \cdot
\right\rangle \colon \mathscr{B}_2 \times \mathscr{B}_2 & \mapsto
L^2(\mathscr{A}\otimes \mathscr{A}, \varphi\otimes \varphi) \notag \\
   \left\langle U, V \right\rangle &=
 \int_{\R_+}^{} U_s \sharp V_s^{*}ds.
\end{align}

\section{Characterizations of freeness}
In this section, we are interested in providing several characterizations of freeness between two multiple integrals. We will derive those characterizations in terms of contractions, covariances and free Malliavin gradients respectively.
\subsection{Characterization in terms of contractions}
Recall the well-known characterization of independence of
multiple Wiener-It\^o integrals by {\"U}st{\"u}nel and Zakai \cite{ustunel_independence_1989} in terms of the first contraction of the associated kernels.
\begin{theorem}[{\"U}st{\"u}nel and Zakai \cite{ustunel_independence_1989}, 1989]
\label{ustuzakaitheorem}
Let $n,m$ be natural numbers and let $f \in
L^2\left(\mathbb{R}_{+}^{n} \right) $ and $g \in
L^2\left(\mathbb{R}_{+}^{m} \right) $ be symmetric functions. Then,
$I_{n}^{W}\left(f \right)$ and $I_{m}^{W}\left(g \right)$ are
independent if and only if $f \otimes_{1} g = 0$ almost everywhere
(for the definition of $\otimes_1$, see the first point of Remark \ref{rrr} below).
\end{theorem}
\begin{remark}\label{rrr}
{\rm
\begin{itemize}
\item In Theorem \ref{ustuzakaitheorem} and throughout the text, the notation
$\otimes_r$ stands for the usual $r$th contraction operator, defined as follows:
if $f\in L^2(\R_+^n)$ and $g\in L^2(\R_+^m)$ are symmetric and if $r\in\{1,\ldots,n\wedge m\}$, we set
\begin{eqnarray*}
&&(f\otimes_r g)(t_1,\ldots,t_{n+m-2r})=\int_{\R_+^r} f(t_1,\ldots,t_{n-r},x_1,\ldots,x_r)\\
&&\hskip3.8cm\times g(t_{n-r+1},\ldots,t_{n+m-2r},x_1,\ldots,x_r)dx_1\ldots dx_r.
\end{eqnarray*}
\item In the context of a multiple Wiener-It\^o integral $I_n^W (f)$, note
that one can always assume without loss of generality that the kernel
$f$ is symmetric, as $I_n^W(f) = I_n^W (\tilde{f})$, where $\tilde{f}$ denotes the symmetrization of the function $f$ given
by 
\begin{equation*}
\tilde{f} \left( x_1, \ldots, x_n \right) = \frac{1}{n!}\sum_{\sigma
  \in \frak{S}_n}^{}f \left( x_{\sigma(1)}, \ldots, x_{\sigma(n)} \right),
\end{equation*}
with $\frak{S}_n$ the symmetric group of $\left\{ 1, \ldots, n \right\}$.
\end{itemize}
}
\end{remark}

A natural question is to ask whether or not the
characterization of independence of {\"U}st{\"u}nel and Zakai has a
counterpart in the free setting. It turns out that a similar
characterization of freeness holds on both the Wigner and the free
Poisson space, which is the first result of this paper.
\begin{theorem}
\label{freeustuzakaiwignercase}
Let $n,m$ be natural numbers and let $f \in L^2\left(\mathbb{R}_{+}^{n} \right) $ and $g \in L^2\left(\mathbb{R}_{+}^{m} \right) $ be symmetric functions. Then, 
\begin{enumerate}[(i)]
\item $I_{n}^{S}\left(f \right)$ and $I_{m}^{S}\left(g \right)$ are free if and only if $f \cont{1} g = 0$ almost everywhere.
\item $I_{n}^{\hat{N}}\left(f \right)$ and $I_{m}^{\hat{N}}\left(g \right)$ are free if and only if $f \star_{1}^{0} g = 0$ almost everywhere.
\end{enumerate}
\end{theorem}
\begin{proof}

First, assume that $I_{n}^{\frak{M}}\left(f \right)$ and $I_{m}^{\frak{M}}\left(g \right)$ are free. Then, by Definition \ref{deffreeness}, it holds that, in particular
\begin{align*}
& \varphi\left( \left[I_{n}^{\frak{M}}\left(f \right)^2 -
    \varphi\left( I_{n}^{\frak{M}}\left(f \right)^2\right)
  \right] \left[I_{m}^{\frak{M}}\left(g \right)^2 - \varphi\left(
      I_{m}^{\frak{M}}\left(g \right)^2\right)   \right]\right) \\
&\qquad\qquad\qquad\qquad = \varphi\left(I_{n}^{\frak{M}}\left(f
  \right)^2I_{m}^{\frak{M}}\left(g \right)^2\right) -
\varphi\left(I_{n}^{\frak{M}}\left(f
  \right)^2\right)\varphi\left(I_{m}^{\frak{M}}\left(g
  \right)^2\right)= 0.
\end{align*}
Observe that 
\begin{align*}
&\varphi\left(I_{n}^{\frak{M}}\left(f \right)^2 I_{m}^{\frak{M}}\left(g \right)^2\right) = \sum_{p=0}^{n}\sum_{r=0}^{m}\varphi\left( I_{2n-2p}^{\frak{M}}\left(f \cont{p} f \right)I_{2m-2r}^{\frak{M}}\left(g \cont{r} g \right)\right) 
\\
&\quad + \mathds{1}_{\left\lbrace \frak{M} = \hat{N}\right\rbrace }\sum_{p=1}^{n}\sum_{r=1}^{m}\varphi\left( I_{2n-2p+1}^{\frak{M}}\left(f \smcont{p} f \right)I_{2m-2r+1}^{\frak{M}}\left(g \smcont{r} g \right)\right)
\\
&= \sum_{p=0}^{n}\sum_{r=0}^{m}\varphi\left( I_{2p}^{\frak{M}}\left(f \cont{n-p} f \right)I_{2r}^{\frak{M}}\left(g \cont{m-r} g \right)\right) \\
& \quad + \mathds{1}_{\left\lbrace \frak{M} = \hat{N}\right\rbrace
}\sum_{p=0}^{n-1}\sum_{r=0}^{m-1}\varphi\left(
  I_{2p+1}^{\frak{M}}\left(f \smcont{n-p} f
  \right)I_{2r+1}^{\frak{M}}\left(g \smcont{m-r} g \right)\right)
  .
\end{align*}
Using the isometry property \eqref{wigneriso}, we get
\begin{align*}
&\varphi\left(I_{n}^{\frak{M}}\left(f \right)^2
  I_{m}^{\frak{M}}\left(g \right)^2\right) = \sum_{p=0}^{n \wedge m} \left\langle f \cont{n-p} f , g \cont{m-p}
  g \right\rangle_{L^2 \left( \R_+^{2p} \right)} \\
& \qquad\qquad\qquad\qquad + \mathds{1}_{\left\lbrace \frak{M} = \hat{N}\right\rbrace }\sum_{p=0}^{(n \wedge m) -1} \left\langle f \smcont{n-p} f , g \smcont{m-p} g \right\rangle_{L^2 \left( \R_+^{2p+1} \right)} 
\\
&= \sum_{p=0}^{n \wedge m}\norm{f \cont{p} g}_{L^2 \left( \R_+^{n+m-2p} \right)}^{2} + \mathds{1}_{\left\lbrace \frak{M} = \hat{N}\right\rbrace }\sum_{p=1}^{n \wedge m}\norm{f \smcont{p} g}_{L^2 \left( \R_+^{n+m-2p+1} \right)}^{2} \\
&= \norm{f}_{L^2 \left( \R_+^{n} \right)}^{2}\norm{g}_{L^2 \left(
    \R_+^{m} \right)}^{2} + \sum_{p=1}^{n \wedge m} \norm{f \cont{p}
  g}_{L^2 \left( \R_+^{n+m-2p} \right)}^{2} \\
& \qquad\qquad\qquad\qquad\qquad\qquad+ \mathds{1}_{\left\lbrace \frak{M} = \hat{N}\right\rbrace }\sum_{p=1}^{n \wedge m}\norm{f \smcont{p} g}_{L^2 \left( \R_+^{n+m-2p+1} \right)}^{2}.
\end{align*}
Recalling that $ \varphi\left(I_{n}^{\frak{M}}\left(f \right)^2 \right) = \norm{f}_{L^2 \left( \R_+^{n} \right)}^{2}$ and $ \varphi\left(I_{m}^{\frak{M}}\left(g \right)^2 \right) = \norm{g}_{L^2 \left( \R_+^{m} \right)}^{2}$ yields
\begin{align}
\label{expressionofcovarianceofthesquares}
& \varphi\left(I_{n}^{\frak{M}}\left(f \right)^2
  I_{m}^{\frak{M}}\left(g \right)^2\right) -
\varphi\left(I_{n}^{\frak{M}}\left(f
  \right)^2\right)\varphi\left(I_{m}^{\frak{M}}\left(g
  \right)^2\right) = \sum_{p=1}^{n \wedge m} \norm{f \cont{p}
  g}_{L^2 \left( \R_+^{n+m-2p} \right)}^{2} \notag \\
&\qquad\qquad\qquad\qquad\qquad\qquad\qquad\qquad + \mathds{1}_{\left\lbrace \frak{M} = \hat{N}\right\rbrace }\sum_{p=1}^{n \wedge m}\norm{f \smcont{p} g}_{L^2 \left( \R_+^{n+m-2p+1} \right)}^{2}.
\end{align}
As the left-hand side of the above equality is zero, the fact that $f \cont{1} g = 0$ a.e. in the Wigner case and $f \star_{1}^{0} g = 0$ a.e. in the free Poisson case follows.

Conversely, assume that $f \cont{1} g = 0$ a.e. in the Wigner case and that $f \star_{1}^{0} g = 0$ a.e. in the free Poisson case. According to Definition \ref{deffreeness} together with the linearity of the functional $\varphi$, we must prove that, for any natural number $\ell$ and for any natural numbers $k_1, \ldots, k_{2\ell}$,
\begin{eqnarray*}
&& \varphi\left( \left[I_{n}^{\frak{M}}\left(f \right)^{k_1} - \varphi\left( I_{n}^{\frak{M}}\left(f \right)^{k_1}\right)   \right]\left[I_{m}^{\frak{M}}\left(g \right)^{k_2} - \varphi\left( I_{m}^{\frak{M}}\left(g \right)^{k_2}\right)   \right] \right.  \\
&& \left. \qquad \cdots \left[I_{n}^{\frak{M}}\left(f \right)^{k_{2\ell-1}} - \varphi\left( I_{n}^{\frak{M}}\left(f \right)^{k_{2\ell-1}}\right)   \right]\left[I_{m}^{\frak{M}}\left(g \right)^{k_{2\ell}} - \varphi\left( I_{m}^{\frak{M}}\left(g \right)^{k_{2\ell}}\right)   \right] \right) =0.
\end{eqnarray*}
\begin{remark}
Observe that we only consider an even number of powers $k$. This comes from the tracial property of the functional $\varphi$ together with the condition that no two adjacent indices $i_j$ can be equal in Definition \ref{deffreeness}. Indeed, if we consider an odd number of powers $k$, we would have
\begin{align*}
& \varphi\left( \left[I_{n}^{\frak{M}}\left(f \right)^{k_1} -
    \varphi\left( I_{n}^{\frak{M}}\left(f \right)^{k_1}\right)
  \right]\left[I_{m}^{\frak{M}}\left(g \right)^{k_2} - \varphi\left(
      I_{m}^{\frak{M}}\left(g \right)^{k_2}\right)\right]\right. \\
& \left. \qquad\qquad\qquad\qquad\qquad\qquad\qquad\qquad \cdots \left[I_{n}^{\frak{M}}\left(f \right)^{k_{2\ell+1}} - \varphi\left( I_{n}^{\frak{M}}\left(f \right)^{k_{2\ell+1}}\right)   \right] \right) \\
&= \varphi\left( \left[I_{n}^{\frak{M}}\left(f \right)^{k_{2\ell+1}} - \varphi\left( I_{n}^{\frak{M}}\left(f \right)^{k_{2\ell+1}}\right)   \right]\left[I_{n}^{\frak{M}}\left(f \right)^{k_1} - \varphi\left( I_{n}^{\frak{M}}\left(f \right)^{k_1}\right)   \right] \right.  \\
& \left.  \qquad\qquad\qquad\quad \left[I_{m}^{\frak{M}}\left(g \right)^{k_2} - \varphi\left( I_{m}^{\frak{M}}\left(g \right)^{k_2}\right)   \right] \cdots \left[I_{m}^{\frak{M}}\left(g \right)^{k_{2\ell}} - \varphi\left( I_{m}^{\frak{M}}\left(g \right)^{k_{2\ell}}\right)   \right]  \right),
\end{align*}
where the first two indices would be the same in the framework of Definition \ref{deffreeness}.
\end{remark}
Let $q < k$ be two non--negative integers. For $0 \leq q \leq k-1$, define the multisets $S_{q}^{k} = \left\lbrace 1,\ldots ,1,0,\ldots ,0 \right\rbrace$ where the element $1$ has multiplicity $q$ and the element $0$ has multiplicity $k-q-1$. Such a set is sometimes denoted $\left\lbrace (1,q),(0,k-q-1) \right\rbrace $. We denote the group of permutations of the multiset $S_{q}^{k}$ by $\mathfrak{S}_{q}^{k}$ and its cardinality is given by the multinomial coefficient $\binom{k-1}{q,m-q-1} = \frac{(k-1)!}{q!(k-q-1)!} = \binom{k-1}{q}$. Observe that in the definition of the group of permutations of a multiset, each permutation yields a different ordering of the elements of the multiset, which is why the cardinality of $\mathfrak{S}_{q}^{k}$ is $\binom{k-1}{q}$ and not $(k-1)!$. Using the Wigner and free Poisson product formulas along with Equation (4.1) in \cite{nourdin_poisson_2013} and Lemma 4.1 in \cite{bourguin_poisson_2015}, we can write
\begin{equation*}
I_{n}^{\frak{M}}\left(f \right)^{k} = \varphi\left(I_{n}^{\frak{M}}\left(f \right)^{k}  \right) + \sum_{r=1}^{kn}I_{r}^{\frak{M}}\left(a_r(f)\right) + \mathds{1}_{\left\lbrace \frak{M} = \hat{N}\right\rbrace }\sum_{r=1}^{kn}I_{r}^{\frak{M}}\left(b_r(f)\right), 
\end{equation*}
where 
\begin{equation*}
a_r(f) = \sum_{\left(p_1, \ldots , p_{k-1} \right) \in A_r }\left( \cdots \left( \left( f \cont{p_1} f\right) \cont{p_2} f\right) \cdots f\right) \cont{p_{k-1}} f
\end{equation*}
with 
\begin{equation*}
A_r = \left\lbrace \left(p_1, \ldots , p_{k-1} \right) \in \left\lbrace 0,1,\ldots,n\right\rbrace^{k-1} \colon kn -2\sum_{i}^{k-1}p_i =r \right\rbrace 
\end{equation*}
and where (recall Definition \ref{defcontractions2} for the contractions appearing below)
\begin{align*}
& b_r(f) = \sum_{q=1}^{k-1}\sum_{\pi \in
  \frak{S}_{q}^{k}}\sum_{\left(p_1, \ldots , p_{k-1} \right) \in
  B_{r,q}^{\pi} }\left( \cdots \left( \left( f \star_{p_{1}}^{p_{1} -
        \pi(1)} f\right) \star_{p_{2}}^{p_{2} - \pi(2)} f\right)
  \cdots  f\right) \\
& \qquad\qquad\qquad\qquad\qquad\qquad\qquad\qquad\qquad\qquad\qquad\qquad\qquad\star_{p_{k-1}}^{p_{k-1} - \pi(k-1)} f
\end{align*}
with, for each $q = 1,\ldots, k-1$ and each $\pi \in \frak{S}_{q}^{k}$,
\begin{equation*}
B_{r,q}^{\pi} = \left\lbrace \left(p_1, \ldots , p_{k-1} \right) \in \bigotimes_{s=1}^{k-1} \left\lbrace \pi(s),\ldots,n \right\rbrace  \colon kn +q - 2\sum_{i}^{k-1}p_i =r \right\rbrace .
\end{equation*}
We get that 
\begin{align*}
& \left[I_{n}^{\frak{M}}\left(f \right)^{k_1} - \varphi\left( I_{n}^{\frak{M}}\left(f \right)^{k_1}\right)   \right]\left[I_{m}^{\frak{M}}\left(g \right)^{k_2} - \varphi\left( I_{m}^{\frak{M}}\left(g \right)^{k_2}\right)   \right] \\
& \qquad \cdots \left[I_{n}^{\frak{M}}\left(f \right)^{k_{2\ell-1}} - \varphi\left( I_{n}^{\frak{M}}\left(f \right)^{k_{2\ell-1}}\right)   \right]\left[I_{m}^{\frak{M}}\left(g \right)^{k_{2\ell}} - \varphi\left( I_{m}^{\frak{M}}\left(g \right)^{k_{2\ell}}\right)   \right]  \\
&= \sum_{r_{1} =1}^{k_{1} n}\sum_{r_{2} =1}^{k_{2} m}\cdots
\sum_{r_{2\ell -1} =1}^{k_{2\ell -1} n}\sum_{r_{2\ell} =1}^{k_{2\ell}
  m}I_{r_{1}}^{\frak{M}}\left(a_{r_{1}}(f) +  \mathds{1}_{\left\lbrace
      \frak{M} = \hat{N}\right\rbrace } b_{r_{1}}(f)\right) \\
& \qquad\qquad I_{r_{2}}^{\frak{M}}\left(a_{r_{2}}(g) +  \mathds{1}_{\left\lbrace
      \frak{M} = \hat{N}\right\rbrace } b_{r_{2}}(g)\right) \cdots
I_{r_{2\ell-1}}^{\frak{M}}\left(a_{r_{2\ell-1}}(f) +
  \mathds{1}_{\left\lbrace \frak{M} = \hat{N}\right\rbrace }
  b_{r_{2\ell-1}}(f)\right) \\
& \qquad\qquad\qquad\qquad\qquad\qquad\qquad\qquad\qquad\qquad\quad I_{r_{2\ell}}^{\frak{M}}\left(a_{r_{2\ell}}(g) +  \mathds{1}_{\left\lbrace \frak{M} = \hat{N}\right\rbrace } b_{r_{2\ell}}(g)\right).
\end{align*}
At this point, observe that the assumptions that $f\cont{1} g =0$ a.e
in the Wigner case and $f \star_{1}^{0} g =0$ a.e in the free Poisson
case imply, by Lemma \ref{lemmacont1impliescontp} and Lemma \ref{lemmacont1impliescontpandcontstarp} respectively, that for any given $ i = 1, \ldots, 2\ell -1$, the contractions between 
$\left(a_{r_{i}}(f) +  \mathds{1}_{\left\lbrace \frak{M} =
      \hat{N}\right\rbrace } b_{r_{i}}(f)\right)$ and
$\left(a_{r_{i+1}}(g) +  \mathds{1}_{\left\lbrace \frak{M} =
      \hat{N}\right\rbrace } b_{r_{i+1}}(g)\right)$ resulting from
using the appropriate product formula iteratively will all be zero
a.e. except for the ones of order zero corresponding to the tensor
product operation (it is the only contraction that can be non-zero
under both the Wigner and free Poisson case assumptions).
\begin{remark}
Note that for the above argument to hold, we need to assume that the functions
$f$ and $g$ are symmetric in order to be able to freely reorder
variables appearing in the contractions of $a_{r_i}(f)$ and
$a_{r_j}(g)$ (as well as in the contractions of $b_{r_{i+1}}(f)$ and
$b_{r_{j+1}}(g)$) so that the assumptions $f \cont{1} g = 0$ a.e. in
the Wigner case and $f \star_{1}^{0} g = 0$ a.e. in the free Poisson
case can be used to deduce that the resulting contractions will all be zero.
\end{remark}
Hence, keeping only the non-zero terms in the above expression yields
\begin{align*}
& \left[I_{n}^{\frak{M}}\left(f \right)^{k_1} - \varphi\left( I_{n}^{\frak{M}}\left(f \right)^{k_1}\right)   \right]\left[I_{m}^{\frak{M}}\left(g \right)^{k_2} - \varphi\left( I_{m}^{\frak{M}}\left(g \right)^{k_2}\right)   \right] \\
& \qquad \cdots \left[I_{n}^{\frak{M}}\left(f \right)^{k_{2\ell-1}} - \varphi\left( I_{n}^{\frak{M}}\left(f \right)^{k_{2\ell-1}}\right)   \right]\left[I_{m}^{\frak{M}}\left(g \right)^{k_{2\ell}} - \varphi\left( I_{m}^{\frak{M}}\left(g \right)^{k_{2\ell}}\right)   \right]  \\
&= \sum_{r_{1} =1}^{k_{1} n}\sum_{r_{2} =1}^{k_{2} m}\cdots
\sum_{r_{2\ell -1} =1}^{k_{2\ell -1} n}\sum_{r_{2\ell} =1}^{k_{2\ell}
  m}I_{r_{1}+ \cdots + r_{2\ell}}^{\frak{M}}\left(\left(a_{r_{1}}(f) +
    \mathds{1}_{\left\lbrace \frak{M} = \hat{N}\right\rbrace }
    b_{r_{1}}(f)\right) \right. \\
& \left. \qquad \otimes\left(a_{r_{2}}(g) +  \mathds{1}_{\left\lbrace
        \frak{M} = \hat{N}\right\rbrace } b_{r_{2}}(g)\right)\otimes
  \cdots \otimes \left(a_{r_{2\ell-1}}(f) +  \mathds{1}_{\left\lbrace
        \frak{M} = \hat{N}\right\rbrace } b_{r_{2\ell-1}}(f)\right)
\right. \\
& \left. \qquad\qquad\qquad\qquad\qquad\qquad\qquad\qquad\qquad\qquad \otimes \left(a_{r_{2\ell}}(g) +  \mathds{1}_{\left\lbrace \frak{M} = \hat{N}\right\rbrace } b_{r_{2\ell}}(g)\right)\right).
\end{align*}
As the quantity $r_{1}+ \cdots + r_{2\ell}$ is strictly positive, applying $\varphi$ to the above expression yields 
\begin{align*}
& \varphi\left( \left[I_{n}^{\frak{M}}\left(f \right)^{k_1} - \varphi\left( I_{n}^{\frak{M}}\left(f \right)^{k_1}\right)   \right]\left[I_{m}^{\frak{M}}\left(g \right)^{k_2} - \varphi\left( I_{m}^{\frak{M}}\left(g \right)^{k_2}\right)   \right] \right.  \\
& \left.  \qquad\quad \cdots \left[I_{n}^{\frak{M}}\left(f \right)^{k_{2\ell-1}} - \varphi\left( I_{n}^{\frak{M}}\left(f \right)^{k_{2\ell-1}}\right)   \right]\left[I_{m}^{\frak{M}}\left(g \right)^{k_{2\ell}} - \varphi\left( I_{m}^{\frak{M}}\left(g \right)^{k_{2\ell}}\right)   \right] \right) =0,
\end{align*}
which is the desired result.
\end{proof}

Observe that the above characterization of freeness is stated and
proven for symmetric kernels only. A natural question is whether or not
this characterization continues to hold in the more general case of a
mirror-symmetric kernel. We provide a negative answer to this
question, proving that our characterization is exhaustive. Concretely,
we will exhibit two mirror-symmetric kernels $f,g\in L^2([0,2]^3)$
such that $\|f \cont{1} g\|_{L^2([0,2]^3)}=0$ but $I_3^S(f)$ and $I_3^S(g)$ are not free.

Indeed, consider $f=\mathds{1}_{[0,1]\times[0,2]\times[0,1]}$ and $g=\mathds{1}_{[1,2]\times[0,2]\times[1,2]}$.
It is readily checked that $f\overset{1}\frown g=0$.
On the other hand, using the product formula \eqref{productformula} iteratively,
we can write
\begin{eqnarray*}
I_3^S(f)^7&=&\sum_{(r_1,\ldots,r_6)\in C} I_{21-2r_1-\ldots-2r_6}^S\bigg(
(((((f\overset{r_1}\frown f) \overset{r_2}\frown f)\overset{r_3}\frown f)\overset{r_4}\frown f)\overset{r_5}\frown f)\overset{r_6}\frown f\bigg)\\
I_3^S(g)^7&=&\sum_{(r_1,\ldots,r_6)\in C} I_{21-2r_1-\ldots-2r_6}^S\bigg((((((g\overset{r_1}\frown g) \overset{r_2}\frown g)
\overset{r_3}\frown g)\overset{r_4}\frown g)\overset{r_5}\frown g)\overset{r_6}\frown g\bigg),
\end{eqnarray*}
where 
\begin{align*}
& C=\left\{(r_1,\ldots,r_6)\in\{0,1,2,3\}^6 \colon r_2\leq 6-2r_1,
\right. \\
& \left. \qquad\qquad\qquad\qquad\qquad\quad r_3\leq
9-2r_1-2r_2,\ldots, r_6\leq 18-2r_1-\ldots-2r_5 \right\}.
\end{align*}
Using the Wigner isometry \eqref{wigneriso}, we deduce that $\varphi
\left(I_3^S(f)^7  \right)=0$ and $\varphi \left(I_3^S(g)^7 \right)=0$, as well as
(the functions $f$ and $g$ being positive)
\begin{align*}
&\varphi \left(I_3^S(f)^7I_3^S(g)^7 \right)\\
&\geq \left\langle  
(((((f\cont{2} f) \cont{2} f)\cont{1} f)\cont{1} f)\cont{1} f)\cont{3}
f,\right. \\
& \left. \qquad\qquad\qquad\qquad\qquad\qquad
(((((g\cont{2}g) \cont{2} g)\cont{1} g)\cont{1} g)\cont{1} g)\cont{3} g
\right\rangle_{L^2([0,2])}\\
&=32\neq 0.
\end{align*}
Consequently, according to the definition of freeness given in
Definition \ref{deffreeness}, $I_3^S(f)$ and $I_3^S(g)$ are not free.
\begin{remark}
The same counterexample would also yield the same conclusion in the free Poisson case
(replacing the Wigner integrals by free Poisson ones) as it is also
the case that $f \star_1^0 g=0$ and as the first
part of the free Poisson product formula
\eqref{productformulafreepoisson} is the same as the Wigner product
formula used above.
\end{remark}

However, even if establishing a characterization of freeness in terms
of contractions in the mirror-symmetric case is not possible, we can
still give a sufficient condition for freeness, which is the object of
the following result.
\begin{theorem}
\label{theofreenesscontmirrorsymcase}
Let $n,m$ be natural numbers and let $f \in
L^2\left(\mathbb{R}_{+}^{n} \right) $ and $g \in
L^2\left(\mathbb{R}_{+}^{m} \right) $ be mirror-symmetric
functions.
\begin{enumerate}[(i)]
\item
If dealing with Wigner integrals, assume that $f^{(\sigma)} \cont{1} g^{(\pi)} = 0$ almost
everywhere for all $\sigma \in \frak{S}_{n}$ and $\pi \in
\frak{S}_{m}$, where 
\begin{equation*}
f^{(\sigma)}\left( x_1, \ldots, x_n \right) = f \left( x_{\sigma(1)},
  \ldots , x_{\sigma(n)} \right),\quad x_1, \ldots, x_n \in \R_+,
\end{equation*}
and a similar definition for $g^{(\pi)}$. Then, $I_{n}^{S}\left(f
\right)$ and $I_{m}^{S}\left(g \right)$ are free.
\item
If dealing with free Poisson integrals, assume that $f^{(\sigma)} \star_1^0 g^{(\pi)} = 0$ almost
everywhere for all $\sigma \in \frak{S}_{n}$ and $\pi \in
\frak{S}_{m}$. Then, one has that $I_{n}^{\hat{N}}\left(f \right)$ and $I_{m}^{\hat{N}}\left(g \right)$ are free.
\end{enumerate}
\end{theorem}
\begin{proof}

Apply the same strategy as in the proof of Theorem
\ref{freeustuzakaiwignercase} with the stronger assumptions.
\end{proof}
\subsection{Characterization in terms of covariances}
The next result is a free analog of \cite[Corollary
5.2]{rosinski_product_1999} by Rosi{\'n}ski and Samorodnitsky, which is itself a consequence of Theorem \ref{ustuzakaitheorem} by {\"U}st{\"u}nel and Zakai.
\begin{corollary}
Let $n,m$ be natural numbers and let $f \in L^2\left(\mathbb{R}_{+}^{n} \right) $ and $g \in L^2\left(\mathbb{R}_{+}^{m} \right) $ be symmetric functions. Then, $I_{n}^{\frak{M}}\left(f \right)$ and $I_{m}^{\frak{M}}\left(g \right)$ are free if and only if their squares are uncorrelated, i.e., if and only if
\begin{equation*}
\mbox{\rm Cov}\left(I_{n}^{\frak{M}}\left(f \right)^2 , I_{m}^{\frak{M}}\left(g \right)^2\right) =0.
\end{equation*} 
\end{corollary}
\begin{proof}

First, assume that $I_{n}^{\frak{M}}\left(f \right)$ and $ I_{m}^{\frak{M}}\left(g \right)$ are free. Then, by Definition \ref{deffreeness}, it holds that 
\begin{align*}
& \varphi\left( \left[I_{n}^{\frak{M}}\left(f \right)^2 -
    \varphi\left( I_{n}^{\frak{M}}\left(f \right)^2\right)
  \right]\left[I_{m}^{\frak{M}}\left(g \right)^2 - \varphi\left(
      I_{m}^{\frak{M}}\left(g \right)^2\right)   \right]\right) \\
& \qquad\qquad\qquad\qquad\qquad = \varphi\left(I_{n}^{\frak{M}}\left(f \right)^2I_{m}^{\frak{M}}\left(g \right)^2\right) - \varphi\left(I_{n}^{\frak{M}}\left(f \right)^2\right)\varphi\left(I_{m}^{\frak{M}}\left(g \right)^2\right) = 0.
\end{align*}
As $\mbox{\rm Cov}\left(I_{n}^{\frak{M}}\left(f \right)^2 , I_{m}^{\frak{M}}\left(g \right)^2\right) =\varphi\left(I_{n}^{\frak{M}}\left(f \right)^2I_{m}^{\frak{M}}\left(g \right)^2\right) - \varphi\left(I_{n}^{\frak{M}}\left(f \right)^2\right)\varphi\left(I_{m}^{\frak{M}}\left(g \right)^2\right)$, the desired conclusion follows.

Conversely, assume that $\mbox{\rm Cov}\left(I_{n}^{\frak{M}}\left(f \right)^2 , I_{m}^{\frak{M}}\left(g \right)^2\right) =0$. Using \eqref{expressionofcovarianceofthesquares}, it holds that 
\begin{align*}
\mbox{\rm Cov}\left(I_{n}^{\frak{M}}\left(f \right)^2 ,
  I_{m}^{\frak{M}}\left(g \right)^2\right) &= \sum_{p=1}^{n \wedge m}
\norm{f \cont{p} g}_{L^2 \left( \R_+^{n+m-2p} \right)}^{2}\\
& \qquad\qquad\qquad\qquad + \mathds{1}_{\left\lbrace \frak{M} = \hat{N}\right\rbrace }\sum_{p=1}^{n \wedge m}\norm{f \smcont{p} g}_{L^2 \left( \R_+^{n+m-2p+1} \right)}^{2},
\end{align*}
which implies that all the contraction norms appearing on the right-hand side of the above equality are zero. In particular, $\norm{f \cont{1} g}_{L^2 \left( \R_+^{n+m-2} \right)}^{2} = 0$ in the Wigner case and $\norm{f \star_{1}^{0} g}_{L^2 \left( \R_+^{n+m-1} \right)}^{2} = 0$ in the free Poisson case, which, by Theorem \ref{freeustuzakaiwignercase} implies that $I_{n}^{\frak{M}}\left(f \right)$ and $I_{m}^{\frak{M}}\left(g \right)$ are free.
\end{proof}
\subsection{Characterization in terms of free Malliavin gradients}
In the context of Wiener integrals, {\"U}st{\"u}nel and Zakai proved
in \cite[Proposition 2]{ustunel_independence_1989} that a necessary
condition for two Wiener integrals $I_n^W \left( f \right)$ and $I_m^W
\left( g \right)$ to be independent was that the inner product of
their Malliavin derivatives was zero almost surely. More precisely,
their statement reads as follows.
\begin{theorem}[{\"U}st{\"u}nel and Zakai
  \cite{ustunel_independence_1989}, 1989]
A necessary condition for the independence of $I_n^W \left( f \right)$ and $I_m^W
\left( g \right)$ is 
\begin{equation}
  \label{malliavincondition}
\left\langle DI_n^W \left( f \right), DI_m^W \left( g \right)
\right\rangle_{L^2 \left( \R_+ \right)} = 0 \quad a.s.
\end{equation}
\end{theorem}

However, they were also able to show that this condition is not
sufficient and hence cannot provide a proper characterization of
independence of Wiener integrals. The technical reason for this is
that this condition implies that only the symmetrization of the first
contraction of $f$ and $g$ be zero almost everywhere, which in turns
does not necessarily imply that the first contraction itself be zero
almost everywhere. As the latter is an equivalent statement to
independence, the sufficiency of \eqref{malliavincondition} fails.

In the free case, a free version of the Malliavin calculus (with
respect to the free Brownian motion) has been
developed by Biane and Speicher in \cite{biane_stochastic_1998}, and
it is a natural question to ask whether it can be used to provide a
characterization of freeness for Wigner integrals.
\begin{remark}
  In this subsection, we only focus on Wigner integrals and not on the
  free Poisson case. The reason for this is that there is no free
  Malliavin calculus available for free Poisson random measures, which
  is what would be needed to explore similar statements in the free
  Poisson case.
\end{remark}

The following result is the main result of this subsection, which is a
characterization of freeness in terms of the free gradient operator
for Wigner integrals with symmetric kernels. It is worth noting that, as opposed to
the case of Wiener integrals studied by {\"U}st{\"u}nel and Zakai, we
are able to provide a positive answer to the question of
characterizing freeness in terms of free gradients, which illustrates
a fundamental difference between the classical case and the free case.
\begin{theorem}
  \label{theoremcharacfreegrad}
Let $n,m$ be natural numbers and let $f \in
L^2\left(\mathbb{R}_{+}^{n} \right) $ and $g \in
L^2\left(\mathbb{R}_{+}^{m} \right) $ be symmetric functions. Then,
$I_{n}^{S}\left(f \right)$ and $I_{m}^{S}\left(g \right)$ are free if
and only if
\begin{equation}
  \label{innerprodiszero}
  \left\langle \nabla I_n^S(f), \nabla I_m^S(g) \right\rangle = 0 
\mbox{ in $L^2(\mathcal{A}\otimes\mathcal{A},\varphi\otimes\varphi)$,}
\end{equation}
where the notation $\left\langle \cdot, \cdot
\right\rangle$ is defined in \eqref{defofinnerprodint}.
\end{theorem}
\begin{proof}

  In the following we will use the shorthand $f^{(k)}_s$ to denote the function
given by
\begin{equation*}
  f_s^{(k)}(x_1,\dots,x_{n-1}) =
f(x_1,\dots,x_{k-1},s,x_{k+1},\dots,x_n).
\end{equation*}
Applying the definition of the action of $\nabla$ on Wigner integrals,
we get that
\begin{align*}
\left\langle \nabla I_n^S(f), \nabla I_m^S(g) \right\rangle & = \int_{\R_+}^{}(\nabla_s I_n^S(f)) \sharp (\nabla_s I_m^S(g))^{*}ds
\\
&  = \sum_{k=1}^n\sum_{q=1}^m \int_{\R_+}^{}[I_{k-1}^{S}\otimes I_{n-k}^{S}]\left( f_s^{(k)}
\right)\sharp \left([I_{q-1}^{S} \otimes I_{m-q}^{S}]\left( g_s^{(q)} \right)\right)^{*}ds
\\
&  = \sum_{k=1}^n\sum_{q=1}^m
\int_{\R_+}^{}[I_{k-1}^{S}\otimes I_{n-k}^{S}]\left( f_s^{(k)}
\right)\sharp [I_{q-1}^{S}\otimes I_{m-q}^{S}]\left( g_s^{(q)} \right)ds,
\end{align*}
where the last equality follows from the full symmetry of the function
$g$. The biproduct formula \eqref{biintegralmultformula} yields
\begin{align*}
& \left\langle \nabla I_n^S(f), \nabla I_m^S(g) \right\rangle
\\
&  = \sum_{k=1}^n\sum_{q=1}^m
\int_{\R_+}^{}\sum_{p=0}^{(k\wedge q)-1}\sum_{r=0}^{(n-k)\wedge
  (m-q)}  [I_{k+q-2-2p}^{S}\otimes I_{n+m-k-q-2r}^{S}]\left( f_s^{(k)}
  \conts{p,r} g_s^{(q)}\right)ds,
\end{align*}
and by using a Fubini argument, it follows that
\begin{align*}
& \left\langle \nabla I_n^S(f), \nabla I_m^S(g) \right\rangle
\\
&  = \sum_{k=1}^n\sum_{q=1}^m
\sum_{p=0}^{(k\wedge q)-1}\sum_{r=0}^{(n-k)\wedge
  (m-q)} [I_{k+q-2-2p}^{S}\otimes I_{n+m-k-q-2r}^{S}]\left(\int_{\R_+}^{} f_s^{(k)}
  \conts{p,r} g_s^{(q)} ds\right).
\end{align*}
The full symmetry of $f$ and $g$ implies that $f_s^{(k)} = f_s^{(n)}$
for every $1 \leq k \leq n$ and $g_s^{(q)} = g_s^{(1)}$
for every $1 \leq q \leq m$. Hence, using Remark \eqref{remarkbicont},
we get 
\begin{equation*}
\int_{\R_+}^{} f_s^{(k)}
  \conts{p,r} g_s^{(q)} ds = f \cont{p+r+1} g,
\end{equation*}
so that we finally get
\begin{align}
  \label{innerprodgradientfinal}
& \left\langle \nabla I_n^S(f), \nabla I_m^S(g) \right\rangle \notag
\\
&  = \sum_{k=1}^n\sum_{q=1}^m
\sum_{p=0}^{(k\wedge q)-1}\sum_{r=0}^{(n-k)\wedge
  (m-q)}  [I_{k+q-2-2p}^{S}\otimes I_{n+m-k-q-2r}^{S}]\left(f \cont{p+r+1} g\right).
\end{align}
Using the Wigner bisometry
\eqref{wignerbisometry}, we see that the quantity 
\begin{equation*}
\varphi \otimes \varphi \left( \left|\left\langle \nabla I_n^S(f), \nabla I_m^S(g) \right\rangle\right|^2\right)
\end{equation*}
is just a sum with strictly positive
coefficients only involving the contractions norms 
\begin{equation*}
\norm{f \cont{1} g}_{L^2 \left( R_+^{n+m-2} \right)}^2, \norm{f \cont{2}
  g}_{L^2 \left( R_+^{n+m-4} \right)}^2, \ldots, \norm{f \cont{n \wedge
    m} g}_{L^2 \left( R_+^{n+m-2(n\wedge m)} \right)}^2.
\end{equation*}
Formally, we have an equality of the type 
\begin{equation}
  \label{repasasumofcontnorms}
\varphi \otimes \varphi \left( \left|\left\langle \nabla I_n^S(f), \nabla I_m^S(g) \right\rangle\right|^2\right) =
\sum_{u=1}^{n\wedge m}c_u \norm{f \cont{u} g}_{L^2 \left( R_+^{n+m-2u} \right)}^2,
\end{equation}
with $c_u>0$.

Now assume that $I_n^S(f)$ and $I_m^S(g)$ are free. By Theorem
\ref{freeustuzakaiwignercase}, this is equivalent to $f \cont{1}g = 0$
almost everywhere, which by Lemma \ref{lemmacont1impliescontp} implies that $f \cont{p}g = 0$ almost
everywhere for all $1 \leq p \leq n \wedge m$. Using
\eqref{repasasumofcontnorms}, we get \eqref{innerprodiszero}.
\\~\\
Conversely, assume that
\begin{equation*}
\left\langle \nabla I_n^S(f), \nabla I_m^S(g) \right\rangle = 0.
\end{equation*}
Then, we have that
\begin{equation*}
\varphi \otimes \varphi \left( \left|\left\langle \nabla I_n^S(f), \nabla I_m^S(g) \right\rangle\right|^2\right) = 0.
\end{equation*}
This implies that all the norms appearing in the representation
\eqref{repasasumofcontnorms} are zero, and im particular that $f \cont{1} g = 0$ almost everywhere. Using Theorem
\ref{freeustuzakaiwignercase} concludes the proof.
\end{proof}

\section{Characterizations of asymptotic freeness}
In the asymptotic context, the problem of interest is to
find necessary and sufficient conditions for the limits in law of
multiple integrals to be free. It is a much more general problem
compared to before, as limits in law of multiple integrals need not be multiple
integrals themselves.
\subsection{Characterization in terms of contractions}
In the classical case, the following result holds (see \cite[Theorem 3.1]{nourdin_asymptotic_2014}).
\begin{theorem}[Nourdin and Rosi{\'n}ski \cite{nourdin_asymptotic_2014}, 2014]
\label{nourdinrosinskiasymptotic}
Let $n,m$ be natural numbers and let $\left\lbrace f_k \colon k \geq 1\right\rbrace \subset  L^2\left(\mathbb{R}_{+}^{n} \right) $ and $\left\lbrace g_k \colon k \geq 1\right\rbrace \subset  L^2\left(\mathbb{R}_{+}^{m} \right) $ be sequences of symmetric functions. Assume that $\big(I_{n}^{W}\left(f_k \right),I_{m}^{W}\left(g_k \right)\big) \cvlaw (F,G)$ as $k \rightarrow \infty$, where $F,G$ are square integrable random variables with laws determined by their moments. Then, $F$ and $G$ are independent if and only if $f_k \otimes_{p} g_k \converge 0$ in $L^2(\R_+^{n+m-2p})$ for all $p=1, \ldots, n \wedge m$.
\end{theorem}
\begin{remark}
{\rm
The fact that the limiting random variables in the above theorem need
to have laws determined by their moments (a condition that we get
automatically in the free setting) has been later shown in \cite{nourdin_strong_2016} to be not necessary.
On the other hand, observe that 
 the necessary and
sufficient condition for asymptotic independence is not $f_k \otimes_1
g_k \converge 0$ in $L^2(\R_+^{n+m-2})$, as one could have expected in view of Theorem
\ref{ustuzakaitheorem}. This weaker condition is necessary but not
sufficient in the asymptotic case, as pointed out in \cite[Remark
3.2]{nourdin_asymptotic_2014}. In the free case, the same phenomenon
happens in the sense that the condition $f_k \cont{1} g_k \converge 0$
in $L^2(\R_+^{n+m-2})$ (in the Wigner case) and $f_k \star_{1}^{0} g_k
\converge 0$ in $L^2(\R_+^{n+m-2})$  (in the free Poisson case) will prove to be necessary but not
sufficient either, for the same reason. 
}
\end{remark}

The following result in the free case
is hence rather an analog of the stronger results of
\cite{nourdin_strong_2016} instead of those found in
\cite{nourdin_asymptotic_2014}. In Theorem \ref{nourdinrosinskiasymptotic} or in the forthcoming Theorem \ref{asymptoticfreeustuzakaiwignercase}, note that $F$ and $G$ do not need to have the
form of a multiple integral. This implies that sequences of multiple
integrals can be used in order to prove the freeness of general random
variables in $L^2 \left( \varphi \right)$ (provided these random
variables admit approximating sequences of multiple integrals with symmetric kernels).

\begin{theorem}
\label{asymptoticfreeustuzakaiwignercase}
Let $n,m$ be natural numbers and let $\left\lbrace f_k \colon k \geq
  1\right\rbrace \subset  L^2\left(\mathbb{R}_{+}^{n} \right) $ and
$\left\lbrace g_k \colon k \geq 1\right\rbrace \subset
L^2\left(\mathbb{R}_{+}^{m} \right) $ be sequences of symmetric
functions such that
\begin{equation}
  \label{jointconvassumption1}
\big(I_{n}^{\frak{M}}\left(f_k \right),I_{m}^{\frak{M}}\left(g_k
\right)\big) \cvlaw (F,G)
\end{equation}
as $k \rightarrow \infty$, where $F,G$ are random variables in $L^2\left( \mathscr{A},\varphi\right) $. Then, 
\begin{enumerate}[(i)]
\item If $\frak{M} = S$, then $F$ and $G$ are free if and only if $f_k
  \cont{p} g_k \converge 0$ in $L^2(\R_+^{n+m-2p})$ for all $p=1, \ldots, n \wedge m$.
\item If $\frak{M} = \hat{N}$, then $F$ and $G$ are free if and only
  if $f_k \cont{p} g_k \converge 0$ in $L^2(\R_+^{n+m-2p})$ and $f_k
  \smcont{p} g_k \converge 0 $ in $L^2(\R_+^{n+m-2p+1})$  for all $p=1, \ldots, n \wedge m$. 
\end{enumerate}
\end{theorem}
\begin{proof}

First, assume that $F$ and $G$ are free. Then,  it holds that
$\mbox{\rm Cov}\left(F^2 , G^2\right) = 0$. Using
\eqref{expressionofcovarianceofthesquares} along with assumption
\eqref{jointconvassumption1} yields
\begin{align*}
&\mbox{\rm Cov}\left(I_{n}^{\frak{M}}\left(f_k \right)^2 ,
  I_{m}^{\frak{M}}\left(g_k \right)^2\right) =  \sum_{p=1}^{n \wedge
  m} \norm{f_k \cont{p} g_k}_{L^2 \left( \R_+^{n+m-2p} \right)}^{2} \\
&\qquad\qquad + \mathds{1}_{\left\lbrace \frak{M} = \hat{N}\right\rbrace }\sum_{p=1}^{n \wedge m}\norm{f_k \smcont{p} g_k}_{L^2 \left( \R_+^{n+m-2p+1} \right)}^{2} \converge \mbox{\rm Cov}\left(F^2 , G^2\right) = 0, 
\end{align*}
so that for all $p=1, \ldots, n \wedge m$, $f_k \cont{p} g_k \converge 0$ (in the Wigner case) and for all $p=1, \ldots, n \wedge m$, $f_k \cont{p} g_k \converge 0$ and $f_k \smcont{p} g_k \converge 0 $ (in the free Poisson case).

Conversely, assume that, for all $p=1, \ldots, n \wedge m$, $f_k
\cont{p} g_k \converge 0$ (in the Wigner case) or that, for all $p=1,
\ldots, n \wedge m$, $f_k \cont{p} g_k \converge 0$ and $f_k
\smcont{p} g_k \converge 0 $ (in the free Poisson case). As in the
proof of Theorem \ref{freeustuzakaiwignercase} (together with
assumption \eqref{jointconvassumption1}), these conditions imply that, for any natural number $\ell$ and for any natural numbers $k_1, \ldots, k_{2\ell}$,
\begin{align*}
& \varphi\left( \left[I_{n}^{\frak{M}}\left(f_k \right)^{k_1} - \varphi\left( I_{n}^{\frak{M}}\left(f_k \right)^{k_1}\right)   \right]\left[I_{m}^{\frak{M}}\left(g_k\right)^{k_2} - \varphi\left( I_{m}^{\frak{M}}\left(g_k \right)^{k_2}\right)   \right] \right. \\
&  \left. \quad \cdots \left[I_{n}^{\frak{M}}\left(f_k \right)^{k_{2\ell-1}} - \varphi\left( I_{n}^{\frak{M}}\left(f_k \right)^{k_{2\ell-1}}\right)   \right]\left[I_{m}^{\frak{M}}\left(g_k \right)^{k_{2\ell}} - \varphi\left( I_{m}^{\frak{M}}\left(g_k \right)^{k_{2\ell}}\right)   \right] \right) \converge 0,
\end{align*}
which implies that $F$ and $G$ are free as they are determined by their moments.
\end{proof}
\begin{remark}
{\rm
Observe that the only difference between the proofs of Theorem
\ref{freeustuzakaiwignercase} and Theorem
\ref{asymptoticfreeustuzakaiwignercase} is the fact that in the
non-asymptotic case, we have one additional step which states that the
seemingly weaker condition $f \cont{1} g = 0$ a.e. implies that, for
all $p=1, \ldots, n \wedge m$, $f \cont{p} g = 0$ a.e. (in the Wigner
case) and that the condition $f \star_{1}^{0} g = 0$ a.e. implies
that, for all $p=1, \ldots, n \wedge m$, $f \cont{p} g = 0$ and $f
\smcont{p} g =0$ a.e. (in the free Poisson case). Recall that these
implications do not necessarily hold true asymptotically, as pointed
out in \cite[Remark 3.2]{nourdin_asymptotic_2014}. For instance, the
sequence $\left\{ f_k \colon n\geq 1 \right\} \subset L^2 \left(
  \left[ 0,1 \right]^2 \right)$ given by 
\begin{equation*}
f_k = \sqrt{k}\sum_{i=0}^{k-1}\mathds{1}_{\left[ \frac{i}{k}, \frac{i+1}{k} \right]^2}
\end{equation*}
satisfies $f_k \cont{1} f_k \converge 0$ in $L^2(\R_+^{2})$, although $f_k \cont{2}
f_k = 1$ for all $k$. As we directly assume the asymptotic equivalent of the conclusions of these implications, the same arguments as in the proof of Theorem \ref{freeustuzakaiwignercase} yield the desired conclusion in the proof of Theorem \ref{asymptoticfreeustuzakaiwignercase}.
}
\end{remark}

As before with Theorem \ref{theofreenesscontmirrorsymcase}, we can give sufficient conditions for
the asymptotic freeness of $F$ and $G$ whenever the sequences of
multiple integrals have mirror-symmetric kernels instead of symmetric
ones. 
\begin{theorem}
  \label{theofreenesscontmirrorsymcaseasymptotic}
Let $n,m$ be natural numbers and let $\left\{ f_k \colon k \geq 0
\right\} \subset L^2\left(\mathbb{R}_{+}^{n} \right) $ and $\left\{ g_k \colon k \geq 0
\right\} \subset L^2\left(\mathbb{R}_{+}^{m} \right) $ be sequences of
mirror-symmetric
functions. Assume that $\big(I_n^{\frak{M}}(f_k),I_m^{\frak{M}}(g_k)\big)
\cvlaw (U,V)$ and that $f_k^{(\sigma)} \cont{p} g_k^{(\pi)} \to 0$ as $k \to \infty$, for all $p =1, \ldots, n \wedge
m$ and all $\sigma \in \frak{S}_{n}$ and $\pi \in
\frak{S}_{m}$, where $f_k^{(\sigma)}$ and $g_k^{(\pi)}$ are defined as in
Theorem \ref{theofreenesscontmirrorsymcase}. Finally, 
if dealing with free Poisson integrals, assume moreover that
$f_k^{(\sigma)} \star_p^{p-1} g_k^{(\pi)} \to 0$ as $k \to \infty$, for all $p =1, \ldots, n \wedge
m$ and all $\sigma \in \frak{S}_{n}$ and $\pi \in
\frak{S}_{m}$.
Then $U$ and $V$ are free.
\end{theorem}
\begin{proof}

Using the exact same argument as in the proof of Theorem
\ref{freeustuzakaiwignercase}, we can obtain that, for any natural number $\ell$ and for any natural numbers $p_1, \ldots, p_{2\ell}$,
\begin{align*}
& \varphi\left( \left[I_{n}^{\frak{M}}\left(f_k \right)^{p_1} - \varphi\left( I_{n}^{\frak{M}}\left(f_k \right)^{p_1}\right)   \right]\left[I_{m}^{\frak{M}}\left(g_k \right)^{p_2} - \varphi\left( I_{m}^{\frak{M}}\left(g_k \right)^{p_2}\right)   \right] \right.  \\
& \left. \quad \cdots \left[I_{n}^{\frak{M}}\left(f_k
    \right)^{p_{2\ell-1}} - \varphi\left( I_{n}^{\frak{M}}\left(f_k
      \right)^{p_{2\ell-1}}\right)
  \right]\left[I_{m}^{\frak{M}}\left(g_k \right)^{p_{2\ell}} -
    \varphi\left( I_{m}^{\frak{M}}\left(g_k \right)^{p_{2\ell}}\right)
  \right] \right) \converge 0.
\end{align*}
Taking the limit as $ k \to \infty$, we get that
\begin{equation*}
\varphi\left( \left[U^{p_1} - \varphi\left( U^{p_1}\right)   \right]\left[V^{p_2} - \varphi\left( V^{p_2}\right)   \right]  \cdots \left[U^{p_{2\ell-1}} - \varphi\left( U^{p_{2\ell-1}}\right)
  \right]\left[V^{p_{2\ell}} -
    \varphi\left( V^{p_{2\ell}}\right)
  \right] \right) = 0,
\end{equation*}
which concludes the proof.
\end{proof}

\subsection{Characterization in terms of covariances}
Based on Theorem \ref{nourdinrosinskiasymptotic}, Nourdin
and Rosi{\'n}ski obtained the following result that links
component-wise convergence and joint convergence of multiple integrals
(see \cite[Corollary 3.6]{nourdin_asymptotic_2014}). As before, note that in the
following results, the random variables $F$ and $G$ need not have the
form of multiple integrals. This implies that sequences of multiple
integrals can be used in order to prove the freeness of general random
variables in $L^2 \left( \varphi \right)$ (provided these random
variables admit approximating sequences of multiple integrals with symmetric kernels).
\begin{theorem}
Let $n,m$ be natural numbers and let $\left\lbrace f_k \colon k \geq 1\right\rbrace \subset  L^2\left(\mathbb{R}_{+}^{n} \right) $ and $\left\lbrace g_k \colon k \geq 1\right\rbrace \subset  L^2\left(\mathbb{R}_{+}^{m} \right) $ be sequences of symmetric functions such that $I_{n}^{W}\left(f_k \right) \cvlaw F$ and $I_{m}^{W}\left(g_k \right) \cvlaw G$ as $k \rightarrow \infty$, where $F,G$ are square integrable independent random variables with laws determined by their moments. If $$\mbox{\rm Cov}\left(I_{n}^{W}\left(f_k \right)^2 , I_{m}^{W}\left(g_k \right)^2\right) \converge 0,$$ then $\left(I_{n}^{W}\left(f_k \right) ,I_{m}^{W}\left(g_k \right)  \right) \cvlaw (F,G)$, as $k \rightarrow \infty$.
\end{theorem}

In the free case, we obtain the following similar  result.
\begin{theorem}
\label{frenessasymptoticwignerintegrals}
Let $n,m$ be natural numbers and let $\left\lbrace f_k \colon k \geq
  1\right\rbrace \subset  L^2\left(\mathbb{R}_{+}^{n} \right) $ and
$\left\lbrace g_k \colon k \geq 1\right\rbrace \subset
L^2\left(\mathbb{R}_{+}^{m} \right) $ be sequences of symmetric
functions such that $\left(I_{n}^{\frak{M}}\left(f_k
  \right),I_{m}^{\frak{M}}\left(g_k \right)  \right)  \cvlaw (F,G)$ as
$k \rightarrow \infty$. Then, $F$ and $G$ are free if and only if $$\mbox{\rm Cov}\left(I_{n}^{\frak{M}}\left(f_k \right)^2 , I_{m}^{\frak{M}}\left(g_k \right)^2\right) \converge 0.$$
\end{theorem}
\begin{proof}
Combine \eqref{expressionofcovarianceofthesquares} with Theorem
\ref{asymptoticfreeustuzakaiwignercase}.
\end{proof}

\subsection{Characterization in terms of free Malliavin gradients}
It is also possible to characterize asymptotic
freeness in terms of the free gradient quantity appearing in Theorem
\ref{theoremcharacfreegrad}. We offer the following statement.
\begin{theorem}
Let $n,m$ be natural numbers and let $\left\lbrace f_k \colon k \geq 1\right\rbrace \subset  L^2\left(\mathbb{R}_{+}^{n} \right) $ and $\left\lbrace g_k \colon k \geq 1\right\rbrace \subset  L^2\left(\mathbb{R}_{+}^{m} \right) $ be sequences of symmetric functions such that $$\big(I_{n}^{S}\left(f_k \right),I_{m}^{S}\left(g_k \right)\big) \cvlaw (F,G)$$ as $k \rightarrow \infty$, where $F,G$ are random variables in $L^2\left( \mathscr{A},\varphi\right) $. 
Then, $F$
and $G$ are free if and only if
\begin{equation*}
  \left\langle \nabla I_n^S(f_k), \nabla I_m^S(g_k) \right\rangle \converge 0 
\ \mbox{in}\ L^2 \left( \mathscr{A}\otimes \mathscr{A}, \varphi
  \otimes \varphi \right),
\end{equation*}
where the notation $\left\langle \cdot, \cdot
\right\rangle$ is defined in \eqref{defofinnerprodint}.
\end{theorem}
\begin{proof}

Combine the representation \eqref{repasasumofcontnorms} with Theorem \ref{asymptoticfreeustuzakaiwignercase}.
\end{proof}

\section{Transfer principles}
Since the characterizations of freeness we have obtained in Section 3 involve quantities which are similar whatever the context (classical or free, Brownian or Poisson), it is natural to study possible transfer principles from one setting to another one. It is the goal of this section to study these aspects.
\begin{theorem}
Let $n,m$ be natural numbers and let $f \in L^2\left(\mathbb{R}_{+}^{n} \right) $ and $g \in L^2\left(\mathbb{R}_{+}^{m} \right) $ be symmetric functions. Assume that $I_{n}^{\hat{N}}\left(f \right)$ and $I_{m}^{\hat{N}}\left(g \right)$ are free. Then, $I_{n}^{S}\left(f \right)$ and $I_{m}^{S}\left(g \right)$ are free. However, the fact that $I_{n}^{S}\left(f \right)$ and $I_{m}^{S}\left(g \right)$ are free does not necessarily imply that $I_{n}^{\hat{N}}\left(f \right)$ and $I_{m}^{\hat{N}}\left(g \right)$ are free, as illustrated by Example \ref{contreexampleequivfreenesswignerfreepoisson}.
\end{theorem}
\begin{proof}

By Theorem \ref{freeustuzakaiwignercase}, if $I_{n}^{\hat{N}}\left(f \right)$ and $I_{m}^{\hat{N}}\left(g \right)$ are free, then it holds that $f \star_{1}^{0} g = 0$ a.e. Lemma \ref{lemmacont1impliescontpandcontstarp} guarantees that $f \star_{1}^{0} g = 0$ a.e. implies $f \cont{1} g = 0$ a.e. Using Theorem \ref{freeustuzakaiwignercase} again concludes the proof.
\end{proof}
\begin{example}
\label{contreexampleequivfreenesswignerfreepoisson}
Let $T$ be a positive real number and let $f,g \in L^2\left(\mathbb{R}_{+} \right)$ be functions defined by $$f(x) = x\mathds{1}_{\left[0,T \right] }(x)\ \ \mbox{ and}\ \ g(x) = \left( x^2 - \frac{3T}{4}x\right) \mathds{1}_{\left[0,T \right] }(x).$$
Note that $$f \cont{1} g = \left\langle f,g \right\rangle_{L^2\left(\mathbb{R}_{+}\right) } = \int_{0}^{T}x\left( x^2 - \frac{3T}{4}x\right)dx = \int_{0}^{T}\left( x^3 - \frac{3T}{4}x^2\right)dx =0$$ whereas
$$f \star_{1}^{0} g(x) = f(x)\cdot g(x) = \left( x^3 - \frac{3T}{4}x^2\right)\mathds{1}_{\left[0,T \right] }(x) \neq 0.$$
Hence, by Theorem \ref{freeustuzakaiwignercase}, $I_{1}^{S}\left(f \right)$ and $I_{1}^{S}\left(g \right)$ are free but $I_{1}^{\hat{N}}\left(f \right)$ and $I_{1}^{\hat{N}}\left(g \right)$ are not free.
\end{example}

Based on Theorem \ref{ustuzakaitheorem} and Theorem \ref{freeustuzakaiwignercase}, we can obtain the following transfer principles between the Wiener and Wigner chaos.
\begin{proposition}
Let $n,m$ be natural numbers and let $f \in L^2\left(\mathbb{R}_{+}^{n} \right) $ and $g \in L^2\left(\mathbb{R}_{+}^{m} \right) $ be symmetric functions. It holds that $I_{n}^{S}\left(f \right)$ and $I_{m}^{S}\left(g \right)$ are free if and only if $I_{n}^{W}\left(f \right)$ and $I_{m}^{W}\left(g \right)$ are independent.
\end{proposition}
\begin{proof}

Observe that as $f$ and $g$ are symmetric functions, it holds that $f \otimes_{1} g = f \cont{1} g$. Using Theorem \ref{ustuzakaitheorem} and Theorem \ref{freeustuzakaiwignercase} concludes the proof.
\end{proof}
\begin{remark}\label{remarksurlecasclassique} 
{\rm 
In the classical Poisson case, there is no known characterization of
independence in terms of the almost sure nullity of a contraction. By
using similar techniques as the ones used in the proof of Theorem
\ref{freeustuzakaiwignercase} (using the definition of moment
independence in place of the definition of freeness), one can prove
that the condition $f \star_{1}^{0} g = 0$ a.e. implies moment
independence. However, moment independence only implies $\widetilde{f
  \star_{1}^{0} g} = 0$ a.e., which is weaker than $f \star_{1}^{0} g
= 0$ a.e. Summing up, one can prove that the condition $f
\star_{1}^{0} g = 0$ a.e. is sufficient but not necessary and that the condition $\widetilde{f \star_{1}^{0} g} = 0$ a.e. is
necessary but not sufficient (the fact that it is not sufficient is
illustrated by the counterexample provided in \cite[Example 5.3]{rosinski_product_1999}). Also pointed out in \cite[Example 5.3]{rosinski_product_1999} is the fact that the squares of multiple Poisson integrals being uncorrelated does not imply that these multiple integrals are independent. This makes it difficult to establish any independence correspondence or transfer principles between the classical and free Poisson chaos. However, it can be pointed out that the freeness of free Poisson multiple integrals implies the freeness of the corresponding Wigner integrals and the independence of the corresponding Wiener integrals.
}
\end{remark}

Despite the above remark, we can still provide the following partial transfer result.
\begin{corollary}
Let $n,m$ be natural numbers and let $f \in L^2\left(\mathbb{R}_{+}^{n} \right) $ and $g \in L^2\left(\mathbb{R}_{+}^{m} \right) $ be symmetric functions. Assume that $I_{n}^{\hat{N}}\left(f \right)$ and $I_{m}^{\hat{N}}\left(g \right)$ are free. Then, $I_{n}^{\hat{\eta}}\left(f \right)$ and $I_{m}^{\hat{\eta}}\left(g \right)$ are moment independent.
\end{corollary}
\begin{proof}
Assuming that $I_{n}^{\hat{N}}\left(f \right)$ and $I_{m}^{\hat{N}}\left(g \right)$ are free, Theorem \ref{freeustuzakaiwignercase} states that $f \star_{1}^{0} g =0$ a.e., which, as pointed out in Remark \ref{remarksurlecasclassique}, is a sufficient condition for $I_{n}^{\hat{\eta}}\left(f \right)$ and $I_{m}^{\hat{\eta}}\left(g \right)$ to be moment independent. Conversely, if it holds that $I_{n}^{\hat{\eta}}\left(f \right)$ and $I_{m}^{\hat{\eta}}\left(g \right)$ are moment independent and $f \star_{1}^{0} g =0$ a.e., Theorem \ref{freeustuzakaiwignercase} ensures that $I_{n}^{\hat{N}}\left(f \right)$ and $I_{m}^{\hat{N}}\left(g \right)$ are free.
\end{proof}

\section{Auxiliary results}
This last section contains two auxiliary results that have been used along the proof of Theorem \ref{freeustuzakaiwignercase}.
\begin{lemma}
\label{lemmacont1impliescontp}
Let $n,m$ be natural numbers and let $f \in
L^2\left(\mathbb{R}_{+}^{n} \right) $ and $g \in
L^2\left(\mathbb{R}_{+}^{m} \right) $ be mirror-symmetric
functions. Assume furthermore that $f \cont{1} g = 0$ almost
everywhere. Then, for all $p = 1, \ldots, n \wedge m$, it holds that
$f \cont{p} g = 0$ almost everywhere.
\end{lemma}
\begin{proof}

Observe that, for any $p = 1, \ldots, n \wedge m$, 
\begin{align*}
  &f \cont{p} g\left(t_1, \ldots , t_{n+m-2p} \right) \\
  & = \int_{\mathbb{R}_{+}^{p}}f\left(t_1,\ldots, t_{n-p}, s_p, \ldots
    ,s_1 \right) g\left( s_1, \ldots ,s_p, t_{n-p+1}, \ldots, t_{n+m-2p}\right)ds_1 \cdots ds_p \\
&= \int_{\mathbb{R}_{+}^{p-1}}\bigg(
  \int_{\mathbb{R}_{+}}f\left(t_1,\ldots, t_{n-p}, s_p, \ldots ,s_1
  \right) \\
&  \qquad\qquad\qquad\qquad\qquad\qquad\quad g\left( s_1, \ldots ,s_p, t_{n-p+1}, \ldots, t_{n+m-2p}\right)ds_1\bigg) ds_2 \cdots ds_p \\
&= \int_{\mathbb{R}_{+}^{p-1}} f \cont{1} g\left(t_1,\ldots, t_{n-p}, s_p, \ldots ,s_2,s_2,\ldots ,s_p, t_{n-p+1}, \ldots, t_{n+m-2p} \right)  ds_2 \cdots ds_p.
\end{align*}
Using the assumption that $f \cont{1} g = 0$ a.e., we get $f \cont{p} g = 0$ a.e., which concludes the proof.
\end{proof}
\begin{lemma}
\label{lemmacont1impliescontpandcontstarp}
Let $n,m$ be natural numbers and let $f \in
L^2\left(\mathbb{R}_{+}^{n} \right) $ and $g \in
L^2\left(\mathbb{R}_{+}^{m} \right) $ be mirror-symmetric functions.
Assume furthermore that $f \star_{1}^{0} g = 0$ almost
everywhere. Then, for all $p = 1, \ldots, n \wedge m$ and all $r = 2,
\ldots, n \wedge m$, it holds that $f \cont{p} g=0$ and $f \smcont{r}
g=0$ almost everywhere.
\end{lemma}
\begin{proof}

Observe that, for any $p = 1, \ldots, n \wedge m$,
\begin{align*}
  & f \cont{p} g\left(t_1, \ldots , t_{n+m-2p} \right) \\
  & = \int_{\mathbb{R}_{+}^{p}}f\left(t_1,\ldots, t_{n-p}, s_{p}, \ldots ,s_1 \right) g\left( s_1, \ldots ,s_{p}, t_{n-p+1}, \ldots, t_{n+m-2p}\right)ds_1 \cdots ds_{p}\\
&= \int_{\mathbb{R}_{+}^{p}}f \star_{1}^{0}g\left(t_1,\ldots, t_{n-p}, s_{p}, \ldots ,s_1, s_2, \ldots ,s_{p}, t_{n-p+1}, \ldots, t_{n+m-2p}\right)ds_1 \cdots ds_{p}.
\end{align*}
Similarly, it holds that, for any $r = 2, \ldots, n \wedge m$, 
\begin{align*}
  & f \smcont{r} g\left(t_1, \ldots , t_{n+m-2r+1} \right) \\
  & = \int_{\mathbb{R}_{+}^{r-1}}f\left(t_1,\ldots, t_{n-r+1},
    s_{r-1}, \ldots ,s_1 \right) \\
  & \qquad\qquad\qquad\qquad\qquad\qquad g\left( s_1, \ldots ,s_{r-1}, t_{n-r+1}, \ldots, t_{n+m-2r+1}\right)ds_1 \cdots ds_{r-1}\\
& = \int_{\mathbb{R}_{+}^{r-1}}f \star_{1}^{0}g\left(t_1,\ldots,
  t_{n-r+1}, s_{r-1}, \ldots ,s_1,s_2, \ldots ,s_{r-1}, \right. \\
& \left. \qquad\qquad\qquad\qquad\qquad\qquad \qquad\qquad\qquad
  t_{n-r+1}, \ldots, t_{n+m-2r+1}\right)ds_1 \cdots ds_{r-1}.
\end{align*}
Using the assumption that $f \star_{1}^{0}g = 0$ a.e. concludes the proof.
\end{proof}

\begin{acknow*}
The authors wish to thank an anonymous referee for a careful
reading of the manuscript as well as for valuable suggestions and remarks.
\end{acknow*}

\bibliographystyle{acm}
\bibliography{biblio}

\end{document}